\documentclass{article}
\usepackage{graphicx}
\usepackage[utf8]{inputenc}
\usepackage[T1]{fontenc}
\usepackage{aeguill}
\usepackage{amssymb}
\usepackage[utf8]{inputenc}
\usepackage[T1]{fontenc}
\usepackage{aeguill}
\usepackage{amssymb}
\usepackage{graphicx}
\usepackage[a4paper,pdftex=true]{geometry}

\usepackage[round]{natbib}
\usepackage{amsmath}
\usepackage{hyperref}
\usepackage[a4paper,pdftex=true]{geometry}
\usepackage{color}

\newtheorem{theorem}{Theorem}

\newtheorem{corollary}[theorem]{Corollary}

\newtheorem{proposition}[theorem]{Proposition}

\newtheorem{remark}{Remark}

\newcommand{\Vect}[1]{\mathbf{#1}}
\newcommand{\tr}[1]{ {#1}^{\!  T}}
\newcommand{\Mat}[1]{  \underline{\mathbf{#1}}}
\newcommand{\Esp}{  \mathbf{E}}
\newcommand{\ZZ}{\mathbb{Z}}
\newcommand{\RR}{\mathbb{R}}
\newcommand{\I}{\mathbf{1}}

\title{\bf \vspace*{3cm}Expectiles for subordinated Gaussian processes with applications
}

\author{Jean-Fran\c{c}ois Coeurjolly$^{1}$ and Hedi Kortas$^{2}$  \vspace*{1cm} \\
$^{1}$ {\it Corresponding author:} \texttt{Jean-Francois.Coeurjolly@upmf-grenoble.fr} \\
Laboratory Jean Kuntzmann, Department of Statistics, Grenoble University,\\
121 Avenue Centrale, 38040 Grenoble Cedex 9, France.\\
$^{2}$ Higher Institute of Management, Department of Quantitative
Methods, \\ Sousse University, Tunisia.}

\begin{document}

\maketitle

\newpage

\begin{abstract}
In this paper, we introduce a new
class of estimators of the Hurst exponent of the fractional
Brownian motion (fBm) process. These estimators are based on
sample expectiles of discrete variations of a sample path of the
fBm process. In order to derive the statistical properties of the
proposed estimators, we establish asymptotic results for sample
expectiles of subordinated stationary
Gaussian processes with unit variance and correlation function
satisfying $\rho(i)\sim \kappa|i|^{-\alpha}$ ($\kappa\in \RR$)
with $\alpha>0$. Via a simulation study, we demonstrate the
relevance of the expectile-based estimation method and show that
the suggested estimators are more robust to data rounding than
their sample quantile-based counterparts.
\end{abstract}


\section{Introduction}

In the statistic literature, there has been a tremendous interest
in analysis, estimation and simulation issues pertaining to the
fractional Brownian motion (fBm) \citep{Mandelbrot}. This is due to
the fact that the fBm process offers an adequate modeling
framework for nonstationary self-similar stochastic processes with
stationary increments and can be used to model stochastic
phenomena relating to various fields of research. A fractional
Brownian motion (fBm), denoted $\{B_{H}(t),t \in \mathbb{R}\}$ with Hurst
exponent $0 < H <1$, is a zero-mean continuous-time Gaussian
stochastic process whose correlation function satisfies
$\mathbb{E}[B_{H}(t)B_{H}(s)]=\frac{\sigma^2}{2}(|t|^{2H}+|s|^{2H}-|t-s|^{2H})$
for all pairs $(t,s)\in \mathbb{R}\times\mathbb{R}$ and
$\sigma^2=\mathbb{E}(B_{H}(1)^2)$. The fBm is $H$-self-similar
i.e., for all $\alpha>0, B_{H}(\alpha
t)\stackrel{d}{=}\alpha^{H}B_{H}(t)$, where $\stackrel{d}{=}$
means the equality of all its finite-dimensional probability
distributions. The process corresponding to the first-order
increments of the fBm is known as the fractional Gaussian noise
(fGn) whose correlation function $\rho_{H}(i)$ is asymptotically
of the order of $|i|^{2H-2}$ for large lag lengths $i$. In
particular, for $1/2<H<1$, the correlations are not summable, i.e.
$\sum_{i=-\infty}^{+\infty}\rho_{H}(i)=\infty$. This property is
referred to as long-range dependence or long-memory whereas the
case $0<H<1/2$ corresponds to short memory.

Several methods aimed at estimating the Hurst characteristic
exponent or long-memory exponent have been developed. Among these statistical methods
figure the Fourier-based methods such as the Whittle maximum
likelihood estimator (see e.g. \citet{Beran, Robinson}) or the
spectral regression based estimator \citep{Beran}. The wavelet estimators have been
also extensively investigated either with an ordinary least
squares \citep{Flandrin}, a weighted least squares (see e.g.
\citet{Abry2, Abry3, BardLMS00,Soltani}) or a maximum likelihood (see e.g.
\citep{Wornell, Percival}) estimation schemes. \citet{FayMRT09} present a deep analysis of Fourier and wavelet methods. Recently, the
so-called discrete variations techniques (see e.g.\citet{Kent,
Istas, Coeurjolly1}) have been introduced. Within this class of
estimators, \citet{Coeurjolly2} proposed a new method
based on sample quantiles to estimate the Hurst exponent in the
more general setting of locally self-similar gaussian processes.
This estimator has been proven robust when dealing with outliers
\citep{Achard}. The latter, often encountered in real world
applications, can induce a significant estimation bias. Actually,
the advantage of quantiles is that they have a bounded
gross-error-sensitivity \citep{Hampel, Huber} allowing them to cope
efficiently with the problem of outliers. Nevertheless, this is
not the only problem faced when dealing with estimation issues.
Indeed, data rounding is also a serious impediment. Data rounding
is common in finance \citep{Bijwaard, Rosenbaum}, economics
\citep{Williams}, computer science \citep{Matthieu, Bois} and
computational physics \citep{Vilmart} and can lead to several
misinterpretations. Quantiles are unfortunately not robust against
data rounding since their local shift sensitivity is unbounded,
see \citet{Hampel,Huber}. \citet{Newey} have
introduced the so-called expectile which, although similar to
quantile, has a bounded local shift-sensitivity and thus can
handle the rounding issue.

In this paper, we derive a Bahadur-type representation for sample
expectiles of a subordinated Gaussian process with unit variance
and correlation function with hyperbolic decay. This allows us to
investigate the statistical properties of a new discrete
variations estimator of the Hurst exponent of the fBm process. In
constructing this estimator, we rely mainly on the scale and
location equivariance properties of expectiles \citep{Newey}. We
will show via a simulation study the robustness of the proposed
estimator against data rounding.

The remainder of this paper is structured as follows: Section 2
deals with asymptotic properties of sample expectiles for a class
of subordinated stationary Gaussian processes with unit variance
and correlation function satisfying $\rho(i)\sim
\kappa|i|^{-\alpha}$ ($\kappa\in \RR$) with $\alpha>0$. A short
simulation study is conducted to corroborate our theoretical
findings. In Section 3, we discuss a sample expectile-based
estimator of the Hurst exponent and derive its statistical
properties. We then perform a simulation study in order to confirm
the effectiveness of the suggested estimation method.

\section{Expectiles for subordinated Gaussian processes}

\subsection{A few notation}

Given some random variable $Z$ with mean $\mu$, $F_Z$ is referred
to the cumulative distribution function of $Z$ and $\xi_Z(p)$ for
$p\in (0,1)$ to its $p$th quantile. It is well-known that the
$p$th quantile of a random variable $Z$ can be obtained by
minimizing asymmetrically the weighted mean absolute deviation
$$
\xi_Z(p) := argmin_\theta \;\; \Esp \big[ |p- \mathbf{1}_{Z\leq
\theta}| . |Z-\theta|\big].
$$
In order to limit the local shift sensitivity of the $p$th
quantile, \citet{Newey} defined the notion of
expectile denoted by $E_Z(p)$ for some $p \in (0,1)$. Rather than
an absolute deviation (function), a quadratic loss function is
considered:
\begin{equation}\label{EZp}
E_Z(p) := argmin_\theta \;\; \Esp \big[ |p- \mathbf{1}_{Z\leq
\theta}| . (Z-\theta)^2\big].
\end{equation}
We may note that the $50\%$-expectile if nothing else than the
expectation of $Z$. \citet{Newey} argued that
providing $\Esp [Z]<+\infty$, then for every $p\in (0,1)$ the
solution of \eqref{EZp} is unique on the set $I_{F_Z}:=\{x\in \RR:
F_Z(x) \in (0,1)\}$.
The expectile can also be defined as the solution of the equation $\Esp\big[|p-\mathbf{1}_{Z\leq \theta}|. (Z-\theta) \big]=0$.\\
A key property of the expectile is that it is scale and location
equivariant \citep{Newey}. The scale equivariance property means
that for $Y=aZ$ where $a>0$, the $p$th expectile of $Y$ satisfies:
\begin{equation}\label{Eqv1}
E_Y(p)=a E_Z(p)
\end{equation}

The $p$th expectile is location equivariant in the sense that for
$Y=Z+b$ where $b\in\mathbb{R}$, the $p$th expectile of $Y$ is such
that:
\begin{equation}\label{Eqv2}
E_Y(p)=E_Z(p)+b
\end{equation}

Now, let $\Vect{Z}=(Z_1,\ldots,Z_n)$ be a sample of identically
distributed random variables with common distribution $F_Z$, the
sample expectile of order $p$ is defined as:
$$
\widehat{E} \left( p ; \Vect{Z}\right) := argmin_\theta \;\;
\frac1n \sum_{i=1}^n \left| p- \I_{Z_i\leq \theta}\right| .
\left(Z_i-\theta\right)^2.
$$

\subsection{Main result}

In order to derive asymptotic results for Hurst exponent estimates
based on expectiles, we have to provide asymptotic results for
sample expectiles of nonlinear functions of (centered)
subordinated stationary Gaussian processes with variance~1 and
with correlation function decreasing hyperbolically. This will be
the setting of the rest of this section. Let
$\{Y_i\}_{i=1}^{+\infty}$ be such a Gaussian process with
correlation function $\rho(\cdot)$ satisfying $\rho(i)\sim
\kappa|i|^{-\alpha}$ for $\kappa\in \RR$ and $\alpha>0$. Let
$\Vect{Y}=(Y_1,\ldots,Y_n)$ a sample of $n$ observations and
$\Vect{h(Y)}=(h(Y_1),\ldots,h(Y_n))$ its subordinated version for
some measurable function $h$.  We wish to provide asymptotic
results for the sample $p$th expectile defined by
\begin{equation}\label{sampleEZp}
\widehat{E} \left( p ; \Vect{h(Y)}\right) := argmin_\theta \;\;
\frac1n \sum_{i=1}^n \left| p- \I_{h(Y_i)\leq \theta}\right| .
\left(h(Y_i)-\theta\right)^2.
\end{equation}
Since the criterion is differentiable in $\theta$, the sample
$p$th expectile also satisfies the following estimating equation
$\psi_n\left(\widehat{E} \left( p ; \Vect{h(Y)}\right) ;
\Vect{h(Y)}\right)=0$ with
\begin{equation}\label{sampleEZp2}
\psi_n\left( \theta; \Vect{h(Y)} \right) := \frac1n \sum_{i=1}^n
\left| p- \I_{h(Y_i)\leq \theta}\right| .
\left(h(Y_i)-\theta\right).
\end{equation}
In the following, we need the two following additional notation
for $Y\sim \mathcal{N}(0,1)$
\begin{eqnarray*}
\psi_{h(Y)} (\theta; p ) &:=& \Esp \left[ \left| p-\I_{h(Y)\leq \theta} \right|. (h(Y)-\theta) \right] \\
\psi_{h(Y)}^\prime(\theta; p ) &:=& - \Esp \left[ \left| p-\I_{h(Y)\leq \theta} \right|\right] = - p (1-F_{h(Y)}(\theta)) -(1-p) F_{h(Y)}(\theta), \\
\end{eqnarray*}
the latter quantity corresponding to the derivative of
$\psi_{h(Y)}(\cdot,p)$ if it is well-defined. Let us note that the
$p$th expectile of $h(Y)$ satisfies $\psi_{h(Y)}(E_{h(Y)};p)=0$.
We now present the assumption on the function $h$ considered in our asymptotic result. \\

\noindent \textbf{[A(h,p)]} \quad $h(\cdot)$ is a measurable function such that $\Esp h(Y)^2<+\infty$ and such that the function $\psi_{h(Y)}(\cdot,p)$ is continuously differentiable in a neighborhood of $E_{h(Y)}(p)$ with negative derivative at this point. \\

Such an assumption is in particular satisfied under the following
one:

\noindent \textbf{[$\mathbf{A^\prime(h)}$]} \quad $h(\cdot)$ is a measurable function such that $\Esp h(Y)^2<+\infty$, $h$ is not ``flat'', i.e. for all $\theta\in \RR$ the set $\{y \in \RR: h(y)=\theta\}$ has null Lebesgue measure. \\

Indeed, if $h$ satisfies \textbf{[$\mathbf{A^\prime(h)}$]} then
$\psi(\cdot,p)$ is differentiable in $\theta$. And since,
$E_{h(Y)}(p)$ belongs to the set $I_{h(Y)}=\{x \in \RR:
F_{h(Y)}(x) \in (0,1)\}$, $\psi^\prime(E_{h(Y)}(p);p)$ is
necessarily negative. For the purpose of this paper, our main
result will be applied with $h(\cdot)=|\cdot|^\beta$ (with
$\beta>0$) or $h(\cdot)=\log|\cdot|$ which obviously satisfy
\textbf{[$\mathbf{A^\prime(h)}$]}.

The nature of the asymptotic result will depend on the correlation
structure of the Gaussian process and on the Hermite rank,
$\tau(p,\theta)$ of the function
$$
\widetilde\psi(t;p,\theta) := \left| p-\I_{h(t)\leq \theta}
\right|. (h(t)-\theta) - \psi_{h(Y)} (\theta; p ).
$$
We recall that the Hermite rank (see {\it e.g.} \citet{Taqqu}) corresponds to the smallest
integer such that the coefficient in the Hermite expansion of the
considered function is not zero. For the sake of simplicity,
assume that the Hermite rank of this function depends neither on
$\theta$ nor $p$ and denote it simply by $\tau$. Again, this could
be weakened since we believe that the next result could be proved
with the following Hermit rank: $\inf_{\theta\in \mathcal{V}(E_{h(Y);p})}
\tau(p,\theta)$. As an example, the Hermite rank of
$\widetilde\psi(\cdot,p,\theta)$ is $1$ for $h(\cdot)=\cdot$ and
$(p,\theta)\in (0,1)\times \RR$ and $2$ for
$h(\cdot)=|\cdot|^\beta$ ($\beta>0$) or $\log |\cdot|$ for
$(p,\theta)\in (0,1)\times \RR^+\setminus\{0\}$. We now present
our main result stating a Bahadur type representation for the
sample $p$th expectile of a subordinated Gaussian process.

\begin{theorem} \label{thm-bahadur}
Let $\{Y_i\}_{i=1}^{+\infty}$ a (centered) stationary Gaussian
process with variance~1 and correlation function satisfying
$\rho(i)\sim \kappa|i|^{-\alpha}$ ($\kappa\in \RR$), as $|i|\to
+\infty$ with $\alpha>0$ and with a function $h$ satisfying
\textbf{[A(h,p)]}. Let $\Vect{h(Y)}=(h(Y_1),\ldots,h(Y_n))$ a
sample of $n$ observations of the subordinated process, then, for
all $p\in (0,1)$
\begin{equation}\label{bahadur}
\widehat{E}\left( p ; \Vect{h(Y)} \right) - E_{h(Y)}(p) = -
\frac{\psi_n\left( E_{h(Y)}(p);\Vect{h(Y)}
\right)}{\psi^\prime\left( E_{h(Y)}(p);p \right)} \;\; + \;\;
o_P(r_n),
\end{equation}
where the sequence $r_n=r_n(\alpha,\tau)$ is defined by
$$
r_n= \left\{ \begin{array}{lll}
n^{-1/2} & \mbox{if} & \alpha\tau>1 \\
n^{-1/2} \log(n) & \mbox{if} & \alpha\tau=1 \\
n^{-\alpha \tau/2} & \mbox{if} & \alpha\tau>1.
\end{array} \right.
$$
\end{theorem}

\begin{proof}
Let us simplify the notation for sake of conciseness: let
$\widehat{E}=\widehat{E}(p;\Vect{h(Y)})$, $E=E_{h(Y)}(p)$,
$\psi_n(E)= \psi_n(E_{h(Y)}(p);\Vect{h(Y)})$ and
$\psi^\prime(E)=\psi^\prime_{h(Y)}(E_{h(Y)}(p);p)$. The first
thing to note is that the sequence $r_n$ corresponds to the
short-range or long-range dependence characteristic  of the
sequence
$\widetilde{\psi}(h(Y_1);p,\theta),\ldots,\widetilde{\psi}(h(Y_n);p,\theta)$.
More precisely $r_n^2$ corresponds to the asymptotic behavior  of
$\Esp \psi_n(E)^2$. Indeed, if $(c_j)_{j\geq 0}$ denotes the
sequence of the Hermite coefficients of the expansion of
$\widetilde{\psi}(\cdot;p,E)$ in Hermite polynomials (denoted by
$(H_j(t))_{j\geq 0}$ and normalized in such a way that
$E[H_j(Y)H_k(Y)]=j!\delta_{jk}$), we may have using standard
developments on Hermite polynomials (see {\it e.g.} \citet{Taqqu})
\begin{eqnarray}
\Esp \psi_n(E)^2 &=& \frac1{n^2} \sum_{i,j=1}^n \Esp \left[
\widetilde{\psi}(Y_i;p,E) \widetilde{\psi}(Y_j;p,E)
\right] \nonumber\\
&=& \frac1{n^2} \sum_{i,j=1}^n \sum_{k_1,k_2\geq 0}
\frac{c_{k_1}c_{k_2}}{k_1! k_2!} \Esp \left[
H_{k_1}(Y_i) H_{k_2}(Y_j) \right] \nonumber\\
&=& \frac{1}{n^2} \sum_{i,j=1}^n \sum_{k\geq \tau} \frac{c_k^2}{k!} \rho(j-i)^k \nonumber\\
&=& \mathcal{O}\Big( \underbrace{\frac1n \sum_{|i|\leq n}
|\rho(i)|^\tau}_{=:\rho_n}\Big)= \mathcal{O}(r_n^2). \label{rhon}
\end{eqnarray}

Let us define $V_n:=r_n^{-1} (\widehat{E}-E)$ and $W_n(E):=-r_n^{-1} \psi_n(E)/\psi^\prime(E)$. We just have to prove that $V_n-W_n(E)$ converges in probability to 0 as $n\to+\infty$. The proof is based on the application of Lemma~1 of \citet{Ghosh} which consists in satisfying  the two following conditions:\\
$(a)$ for all $\delta>0$, there exists $\varepsilon=\varepsilon(\delta)$ such that $P(|W_n(E)|>\varepsilon)< \delta$.\\
$(b)$ for all $y\in \RR$ and for all $\varepsilon>0$
$$
\lim_{n\to +\infty}P(V_n\leq y, W_n(E)\geq y+\varepsilon) =
\lim_{n\to +\infty}P(V_n\geq y+\varepsilon, W_n(E)\leq y)=0.
$$

$(a)$ is in particular fulfilled if we prove that $\Esp
W_n(E)^2=\mathcal{O}(1)$ which follows from~\eqref{rhon} since
$\Esp W_n(E)^2=\psi^\prime(E)^{-2} r_n^{-2} \Esp \psi_n(E)^2 =
r_n^{-2} \times \mathcal{O}(\rho_n) = \mathcal{O}(1)$.

$(b)$ We consider only the first limit. The second one follows
similar developments. We first state that the map $\psi_n(\cdot)$
is decreasing. Indeed, let $\theta\leq \theta^\prime$ and denote
by $Z_i(\theta)$ the variable $|p-\I_{h(y_i)\leq
\theta)}|.(h(Y_i)-\theta)$. If $h(Y_i) \leq \theta$ or $h(Y_i) >
\theta$, we obviously get $Z_i(\theta)>Z_i(\theta^\prime)$ a.s.
leading to the decreasing of $\psi_n(\cdot)$. And in the in between
case, $Z_i(\theta)-Z_i(\theta^\prime)= p(\theta^\prime-\theta) +
\theta^\prime-h(Y_i) \geq 0$ which leads to the same conclusion.
Let $y\in \RR$, then also using the fact that
${\psi}_n(\widehat{E})= \psi(E)=0$ and $\psi^\prime(E)<0$, we
derive
\begin{eqnarray*}
\{V_n\leq y \} &=& \{ \widehat{E} \leq y\times r_n + E  \} \\
&=& \{ \psi_n(\widehat{E}) \geq\psi_n(y \times r_n +E) \} \\
&=& \{ \psi(y \times r_n +E) - \psi_n(y\times r_n+E) \geq \psi(y\times  r_n+E) -\psi(E) \} \\
&=& \{ W_n(y \times r_n +E) \leq y_n \},
\end{eqnarray*}
where $y_n=r_n^{-1} \psi^\prime(E)^{-1} (\psi(y \times
r_n+E)-\psi(E))$. Under the assumption \textbf{[A(h,p)]}, $y_n \to
y$ as $n\to +\infty$. Now, let $U_n:= \psi^\prime(E) \left( W_n(E)
-W_n(y \times r_n+E) \right)$, explicitly given by
$$
U_n = \frac1{nr_n} \sum_{i=1}^n \left(\widetilde \psi(
Y_i;p,E+y\times r_n)- \widetilde \psi( Y_i;p,E) \right).$$ Let
$c_{j,n}$ the $j$th Hermite coefficient of the function
$r_n^{-1}\left(\widetilde\psi(t;p,E+y\times
r_n)-\widetilde\psi(t;p,E) \right)$, then under the assumption
\textbf{[A(h,p)]} and from the dominated convergence theorem we
can prove that
$$
c_{j,n} \stackrel{n\to+\infty}{\longrightarrow} y \; \Esp \left[
|p-\I_{h(Y)\leq E}| H_j(Y) \right] =: \widetilde c_j.
$$
Therefore for $n$ large enough,
\begin{eqnarray*}
\Esp[U_n^2] &=& \frac1{n^2} \sum_{i,j=1}^n \sum_{k\geq \tau} \frac{c_{k,n}^2}{k!}\rho(j-i)^k \\
&\leq & \frac2n \sum_{|i|\leq n}\sum_{k\geq \tau} \frac{\widetilde{c}_k^2}{k!} \rho(i)^\tau \\
&=& \mathcal{O}(\rho_n) = \mathcal{O}(r_n^2)
\end{eqnarray*}
which leads to the convergence of $U_n$ to 0 in probability. For
all $\varepsilon>0$, there exists $n_0(\varepsilon)$ such that for
all $n\geq n_0(\varepsilon)$, $y_n\leq y + \varepsilon/2$.
Therefore for $n\geq n_0(\varepsilon)$
\begin{eqnarray*}
P(V_n\leq y \;,\; W_n\geq y+\varepsilon)& =& P(W_n(y \times r_n +E) \leq y_n\;,\; W_n \geq y+\varepsilon) \\
&\leq & P(W_n(y\times r_n +E)\leq y+\varepsilon/2\;, \;W_n(E)\geq y+\varepsilon) \\
&\leq& P \left( \left| W_n(y\times r_n + E) -W_n(E)\right|\geq \varepsilon/2\right) \\
& \stackrel{n\to +\infty}{\to} & 0,
\end{eqnarray*}
which ends the proof.
\end{proof}

In the case of short-range dependence, i.e. $\alpha\tau>1$ then,
using the Bahadur type representation of expectiles, we derive
immediately the following asymptotic normality for the sample
expectile and some generalisations. This result is based on
standard central limit theorem for means of subordinated Gaussian
stationary processes \citep{Taqqu, Arcones}.Therefore the proof is
omitted.

\begin{corollary} \label{cor-bahadur}${ }$\\
$(i)$ Under the assumptions of Theorem~\ref{thm-bahadur} with
$p\in (0,1)$ and $\alpha \tau >1$, then as $n\to + \infty$
$$
\sqrt{n} \left( \widehat{E}\left( p ; \Vect{h(Y)} \right) -
E_{h(Y)}(p) \right) \stackrel{d}{\longrightarrow}
\mathcal{N}(0,\sigma^2(p)),$$ where
$$
\sigma^2(p) = \frac{1}{\psi^\prime\left( E_{h(Y)}(p);p \right)^2}
\; \sum_{i\in \ZZ} \sum_{k\geq \tau} \frac{c_k(p)^2}{k!} \rho(i)^k
$$
and where $c_k(p)$ is the $k$th Hermite coefficient of the expansion of the function $\psi(h(\cdot); E_{h(Y)}(p);p)$ in Hermite polynomials.\\
$(ii)$ Let $\{Y_i^1\}_{i=1}^{+\infty}$ and
$\{Y_i^2\}_{i=1}^{+\infty}$ two (centered) stationary Gaussian
processes with variances~1 and correlation functions (resp.
cross-correlation functions) $\rho^1,\rho^2$ (resp.
$\rho^{12},\rho^{21}$) decreasing hyperbolically with exponents
$\alpha^1,\alpha^2$ (resp. $\alpha^{12},\alpha^{21}$). Let $p \in
(0,1)$, $h$ a function satisfying \textbf{[A(h,p)]} and let
$\Vect{h(Y^1)}$ and $\Vect{h(Y^2)}$ be the samples of $n$
observations of the two subordinated samples. If
$\min(\alpha^1,\alpha^2,\alpha^{12},\alpha^{21})\times \tau>1$,
then as $n\to +\infty$
$$\sqrt{n} \left( \widehat{E}\left( p ; \Vect{h(Y^1)} \right) - E_{h(Y)}(p) , \widehat{E}\left( p ; \Vect{h(Y^2)} \right) - E_{h(Y)}(p)\right)^T
\stackrel{d}{\longrightarrow} \mathcal{N}(0,\Mat{\Sigma}).
$$
where $\Mat{\Sigma}$ is the $(2,2)$ matrix with entries
$\Sigma_{ab}$ for $a,b=1,2$ given by
\begin{equation}\label{def-Sigma}
\Sigma_{ab} = \frac1{\psi^\prime\left( E_{h(Y)}(p);p \right)^2} \;
\sum_{i\in \ZZ} \sum_{k\geq \tau} \frac{c_k(p)^2 }{k!}
\rho^{ab}(i)^k.
\end{equation}
\end{corollary}
As it was established for sample quantiles \citep{Coeurjolly2}, a
non standard limit towards a Rosenblatt process is expected in the
other cases ($\alpha\tau\leq 1$). This case is not considered here.

\subsection{Simulations}

To illustrate a part of the previous results, we propose a short
simulation study in this section. The latent stationary Gaussian
process we consider here is the fractional Gaussian noise with
variance~1, which is obtained by taking the discretized increments
from a fractional Brownian motion. The correlation function of the
fractional Gaussian noise with Hurst parameter (or self-similarity
parameter) $H\in (0,1)$ satisfies the hyperbolic decreasing
property required in Theorem~\ref{thm-bahadur} with $\alpha=2-2H$.
Discretized sample paths of fractional Brownian motion can be
generated exactly using the embedding circulant matrix method
popularized by \citet{WoodChan} (see also \citet{Coeurjolly0}) which
is implemented in the \texttt{R} package \texttt{dvfBm}.

Figures~\ref{fig1-bplot} and \ref{fig2-bplot} illustrate the
convergence of the sample expectiles. Three $h$ functions are
considered: $h(\cdot)=(\cdot)$, $(\cdot)^2$ and $\log|\cdot|$. The
related Hermite rank of the function $\widetilde{\psi}$ is
respectively 1,2 and 2 for these three $h$ functions. The sample
size of the simulation is fixed to $n=500$. We can claim the
convergence of the sample expectile $\widehat{E} (p;{h(Y)})$
towards $E_{h(Y)}(p)$ for all  the values of $\alpha$ (or $H$), p
and for the three functions $h$ considered. If we focus on
$h(\cdot)=(\cdot)$, we can also remark a higher variance of the
sample estimates for $\alpha=0.6$ compared to $\alpha=1.4$. This
is in agreement with the theory since for $\alpha=0.6$,
$\alpha\tau=0.6<1$ and the rate of convergence is lower than
$n^{-1/2}$ which means an increasing of the variance. For the two
other functions considered, then $\alpha \tau$ is always greater
than 1 (it equals either 2.8 or 1.2 in our simulations) and we do
not observe such an increasing of the variance.

To put emphasis on this last point, Figure~\ref{fig-logvar} shows
in log-scale the average (over the 9 order of expectiles
considered in the simulation, i.e. $p=0.1,\ldots,0.9$) of the
empirical variances in terms of $n$ for the three $h$ functions
and for the two values of $\alpha=0.6$ and $\alpha=1.4$. We
clearly observe that as soon as $\alpha \tau>1$, the slope of the
curves is close to $-1$ which agrees with the result presented in
Corollary~\ref{cor-bahadur} for example. When $h(\cdot)=1$ and
$\alpha=0.6$, we observe that the slope is about $-0.6$ which
seems to agree with the convergence in $n^{-\alpha \tau}$ which is
expected from Theorem~\ref{thm-bahadur}.

\begin{figure}[htbp]
\begin{tabular}{cc}
\includegraphics[scale=.17]{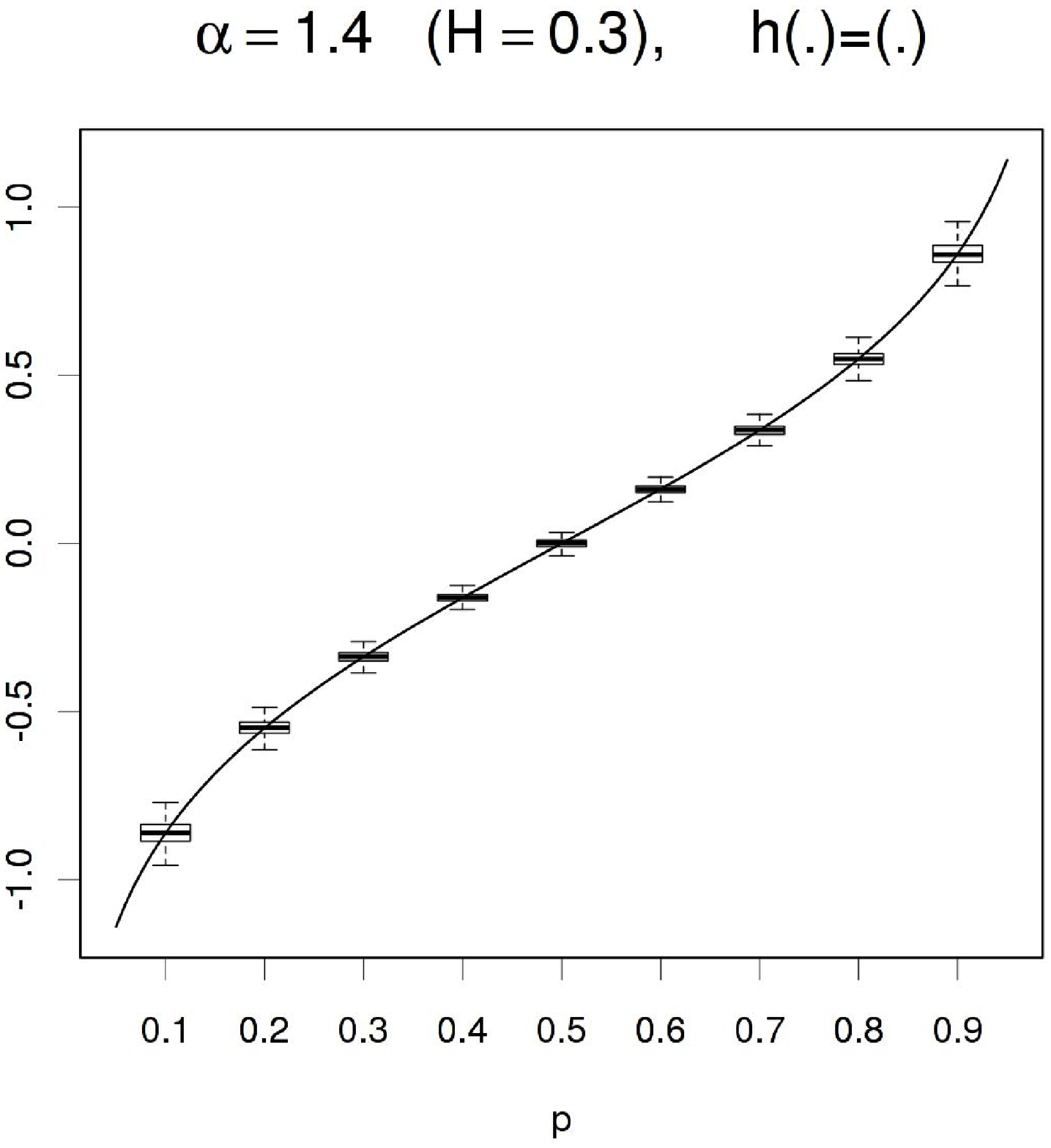} & \includegraphics[scale=.17]{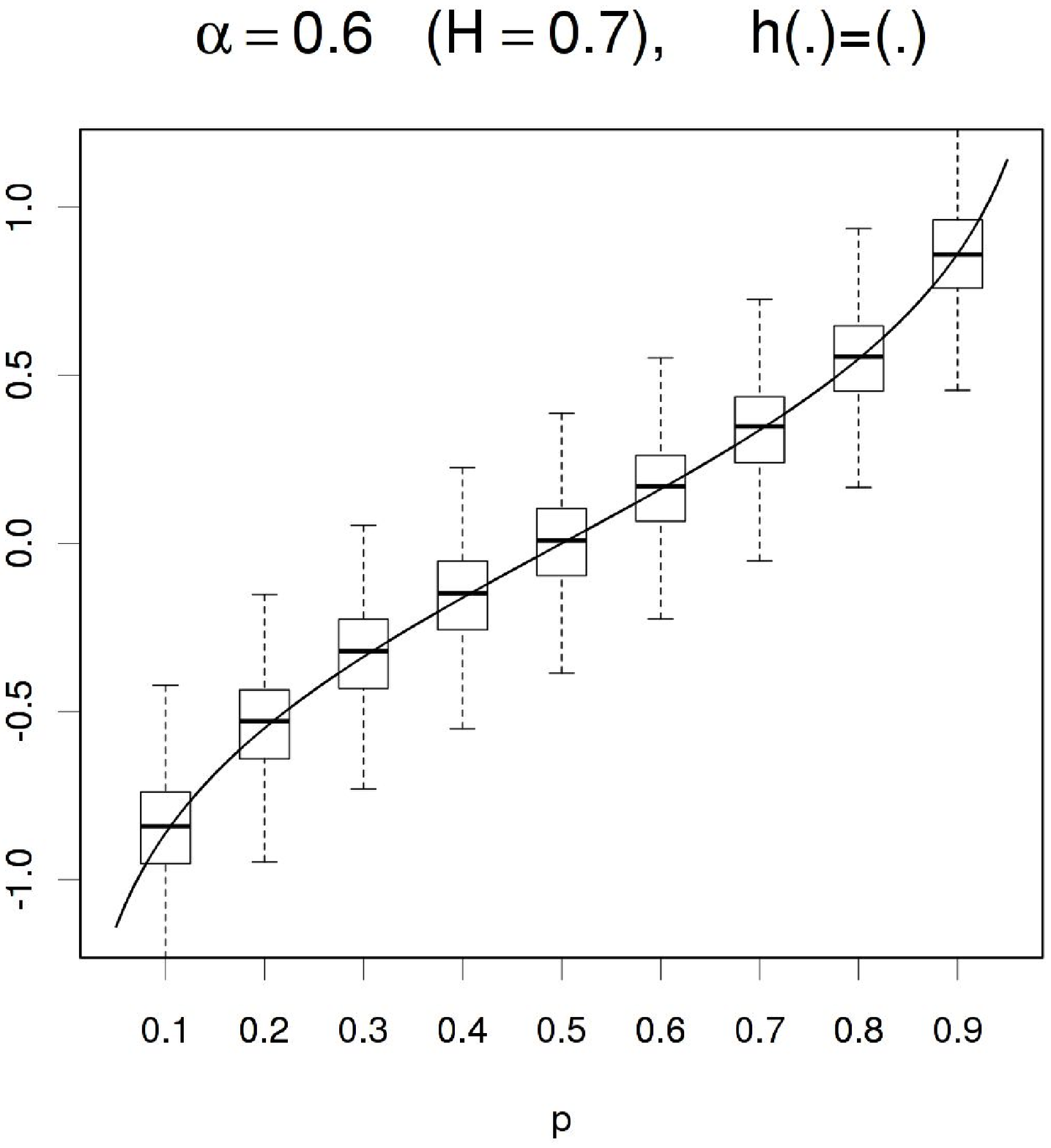} \\
\end{tabular}
 \caption{Boxplots of sample expectiles for
expectiles of order $p=0.1,\ldots,0.9$ based on $m=500$
replications of fractional Gaussian noise with length $n=500$ and
with Hurst parmeter $H=0.3$ (left, $\alpha=1.4$) and $H=0.7$
(right, $\alpha=0.6$). The $h$ functions considered here is the
identity function (with Hermite rank~1). The curves correspond to
the theoretical expectile functions for $Y\sim
\mathcal{N}(0,1)$.\label{fig1-bplot}}
\end{figure}

\begin{figure}[htbp]
\begin{tabular}{cc}
\includegraphics[scale=.17]{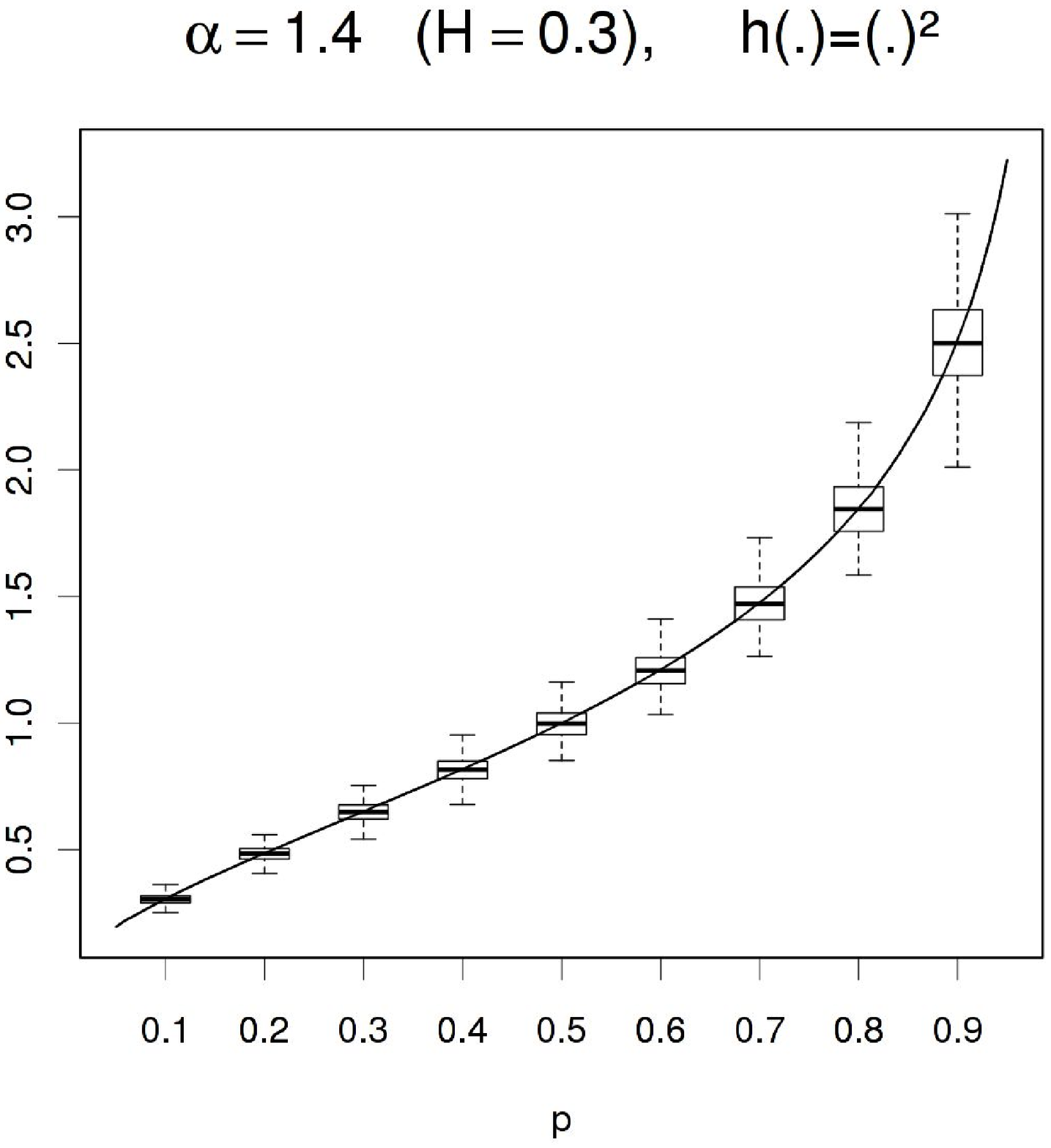} & \includegraphics[scale=.17]{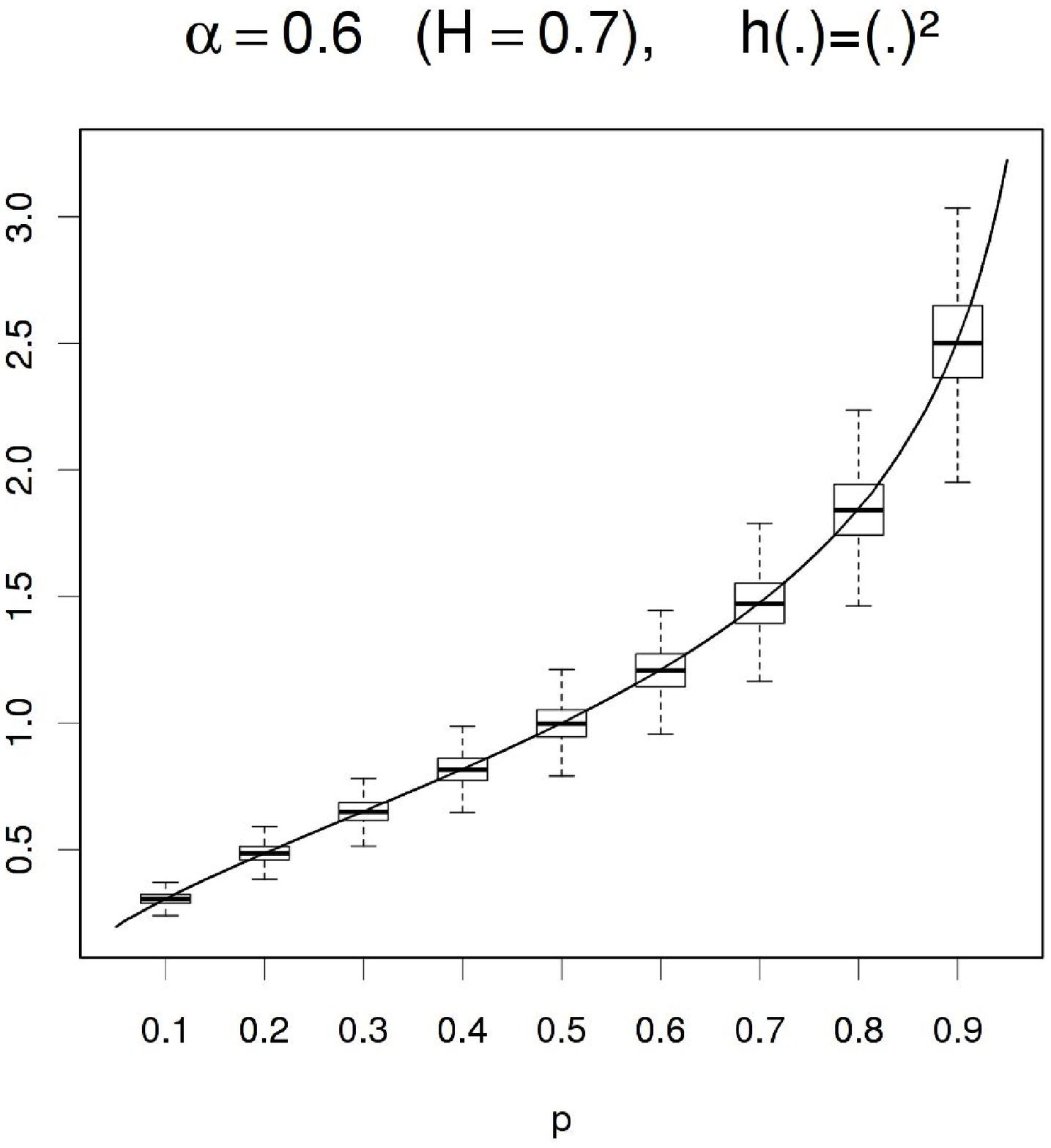} \\
\includegraphics[scale=.17]{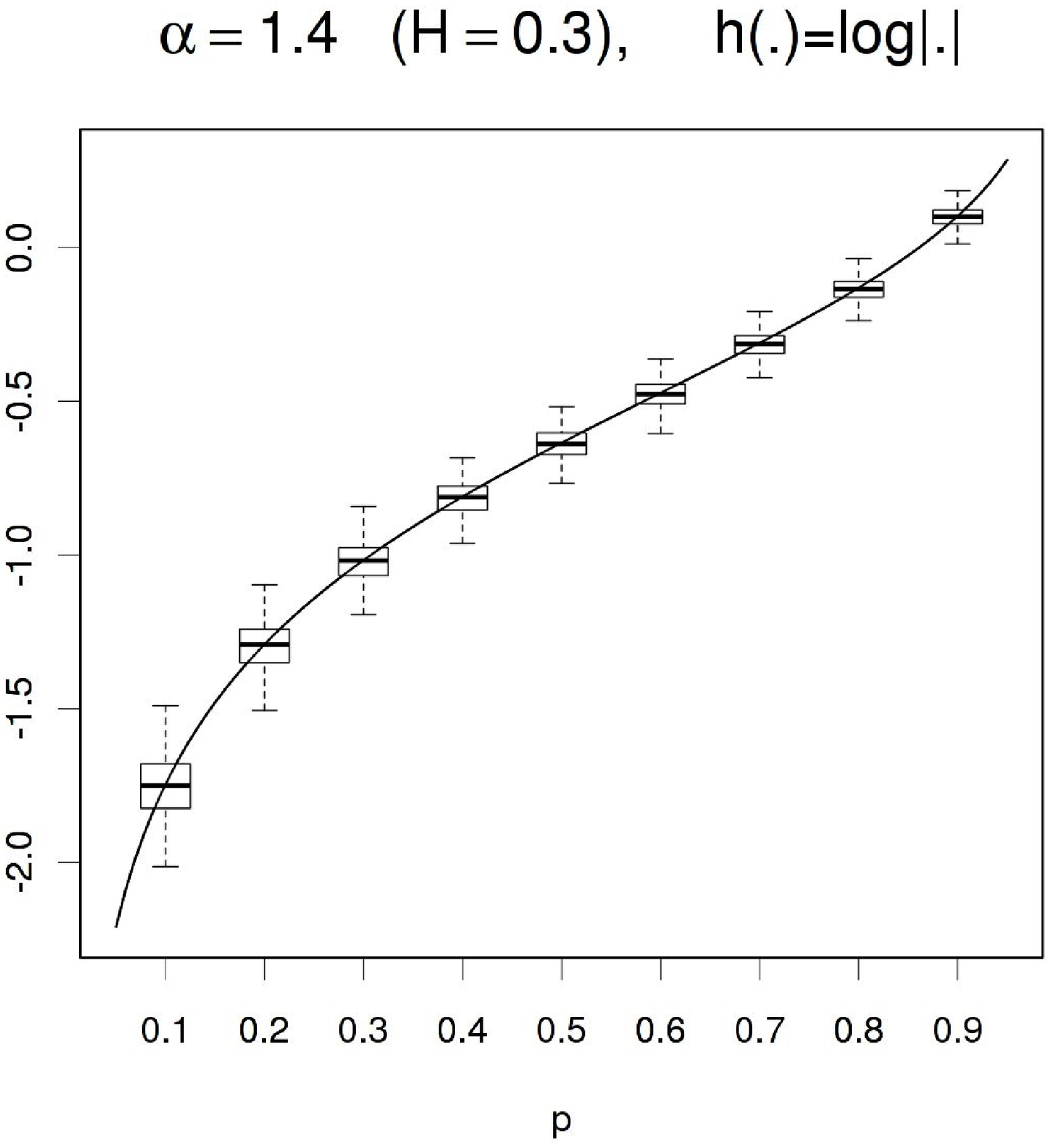} & \includegraphics[scale=.17]{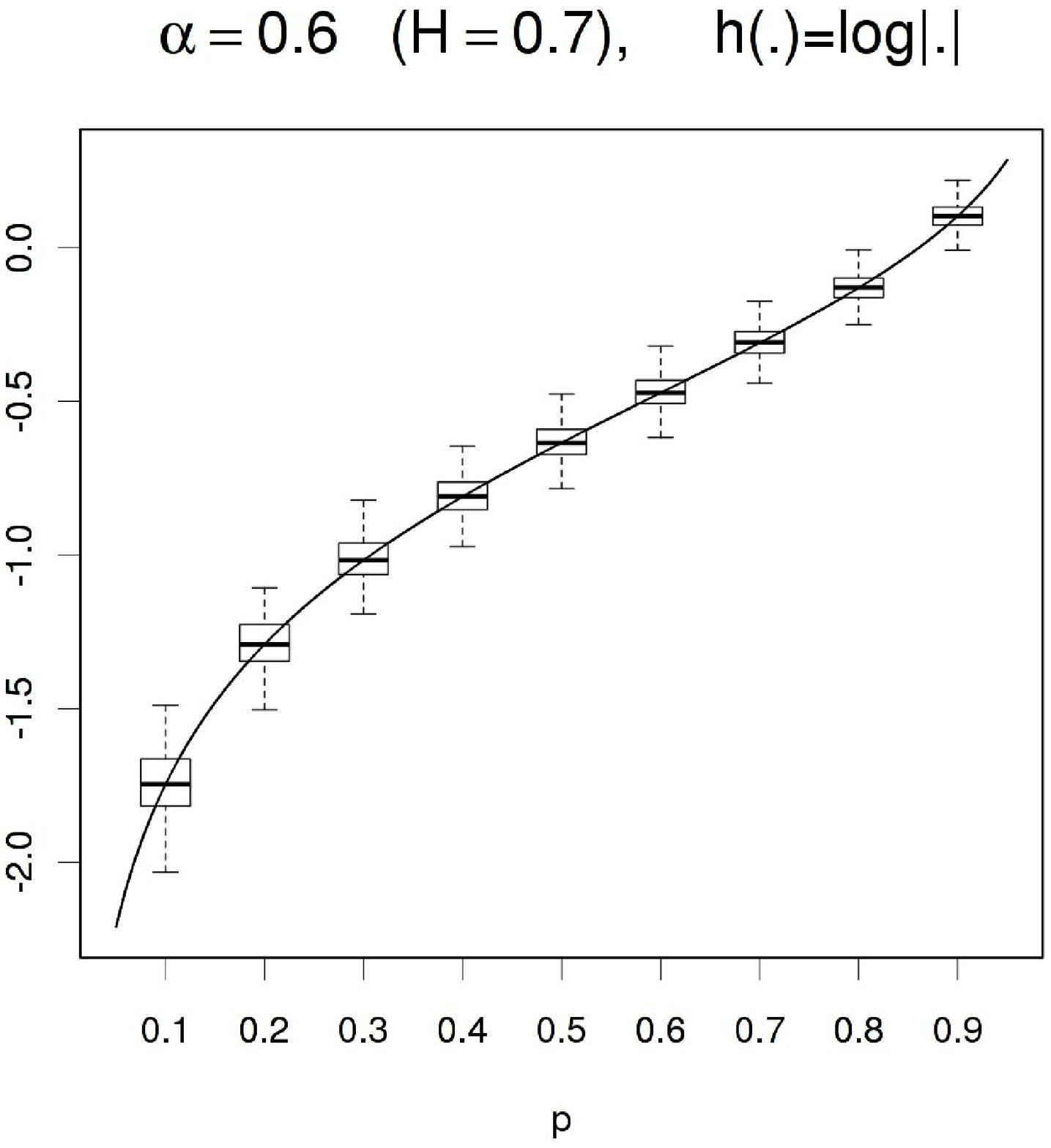}
\end{tabular}
 \caption{Boxplots of sample expectiles for
expectiles of order $p=0.1,\ldots,0.9$ based on $m=500$
replications of fractional Gaussian noise with length $n=500$ and
with Hurst parmeter $H=0.3$ (left, $\alpha=1.4$) and $H=0.7$
(right, $\alpha=0.6$). Two $h$ functions with Hermite rank 2 have
been considered here: $h(\cdot)=(\cdot)^2$ (top) and
$h(\cdot)=\log|\cdot|$ (bottom). The curves correspond to the
theoretical expectile functions for $Y^2$ (middle) and $\log|Y|$
(bottom) where $Y\sim \mathcal{N}(0,1)$.\label{fig2-bplot}}
\end{figure}

\begin{figure}[htbp]
\begin{tabular}{cc}
\includegraphics[scale=.17]{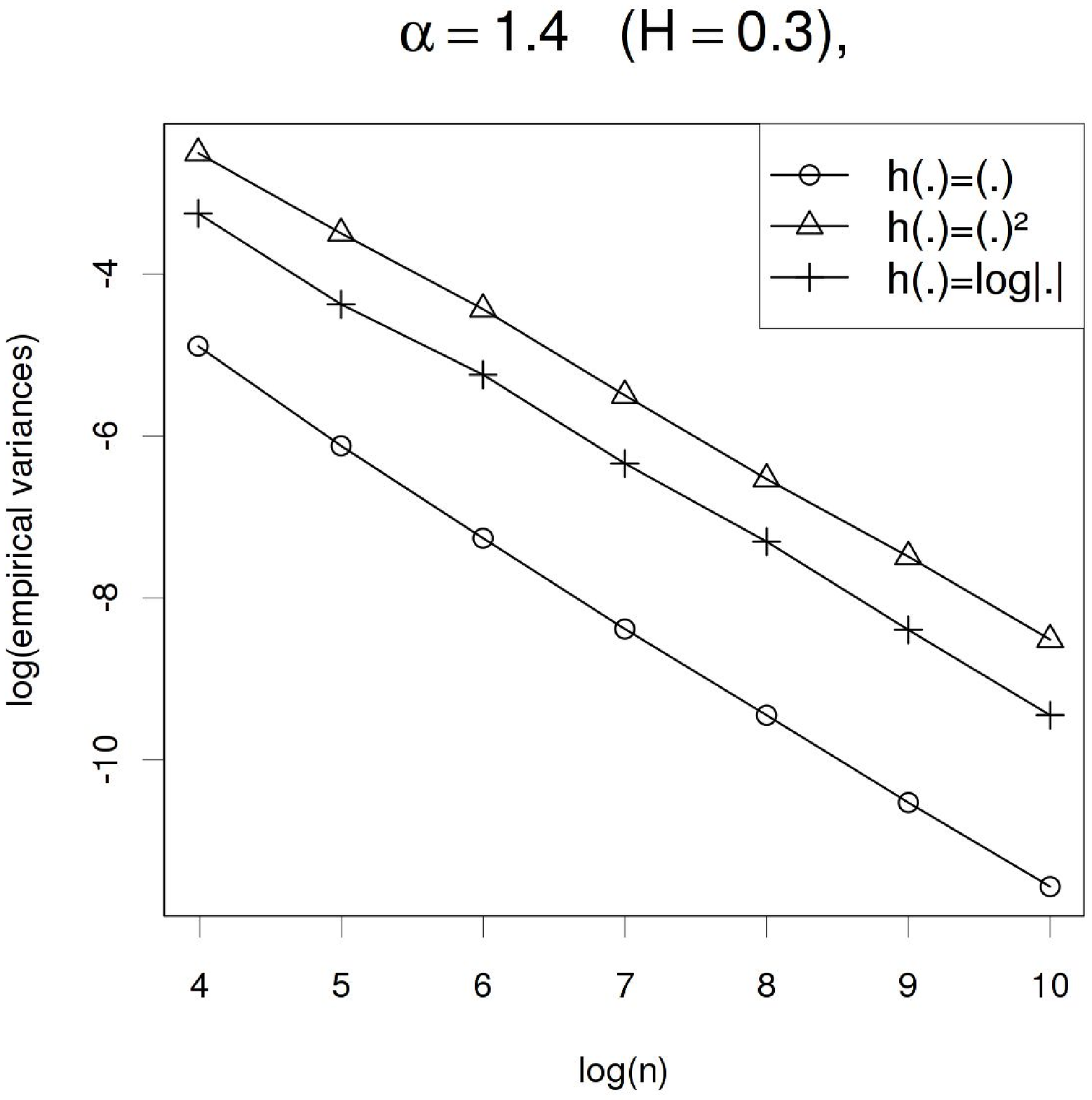} & \includegraphics[scale=.17]{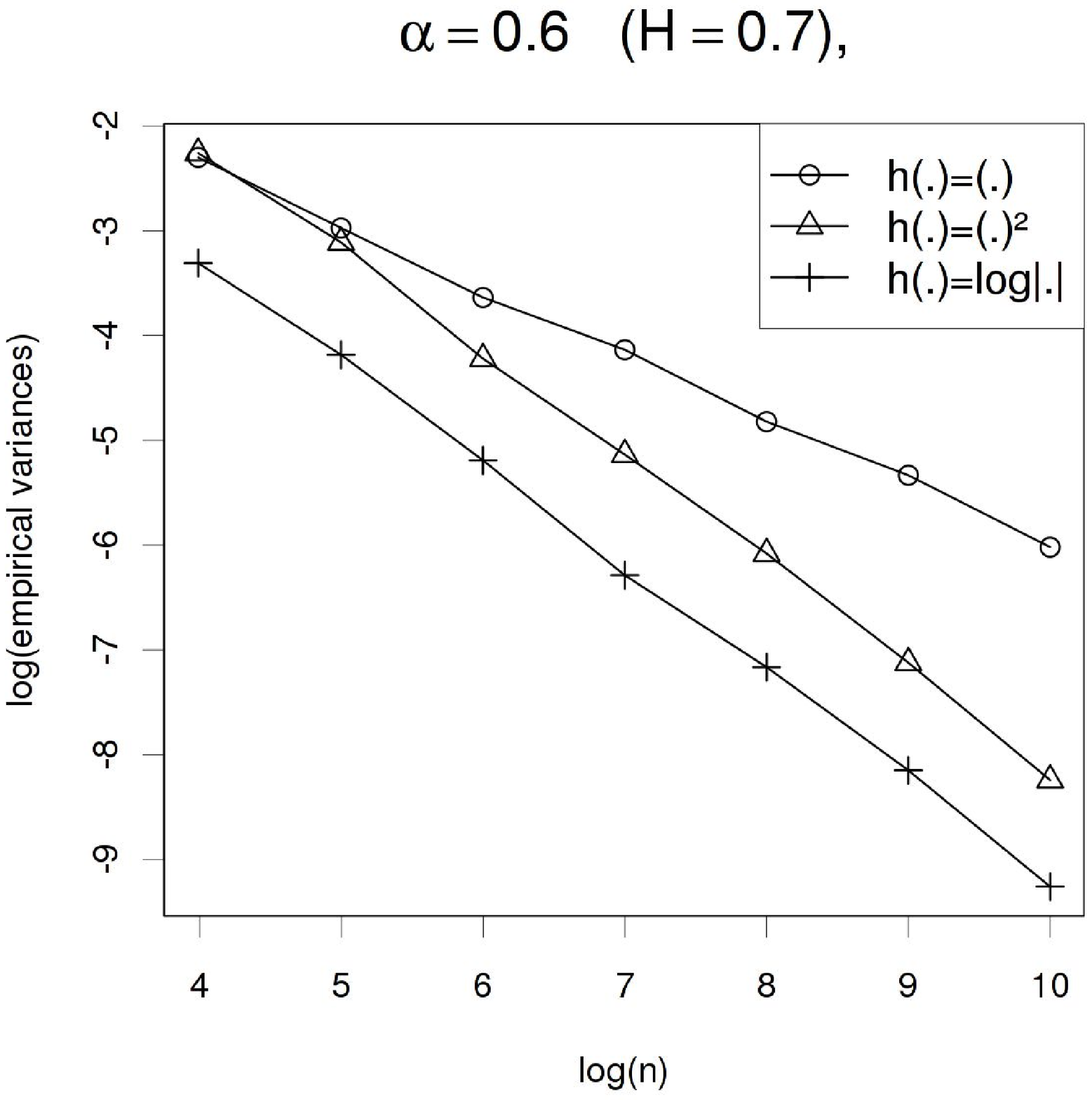} \\
\end{tabular}
 \caption{Means of empirical variances of sample expectiles in
terms of $n$ in log-scale based on $m=500$ replications of
fractional Gaussian noise with parameters $H=0.3$ (left,
$\alpha=1.4$) and $H=0.7$ (right, $\alpha=0.6$). More precisely,
we consider the vector of probability $(0.1,\ldots,0.9)$ for the
orders of the expectiles and we compute $\widehat{\sigma}_n^2=1/9
\times \sum_{i=1}^{9} \widehat{\sigma}_{i,n}^2$ where
$\widehat{\sigma}^2_{i,n}$ is the empirical variance for the
expectile with order $i/10$ for the sample size $n$. Three choices
of $h$ functions have been considered:
$h(\cdot)=(\cdot),(\cdot)^2$ and $\log|\cdot|$.\label{fig-logvar}}
\end{figure}

\section{Estimation of the Hurst exponent using sample expectiles and discrete variations}

\subsection{Estimation method and asymptotic results}

Let $\Vect{X}=\left(X(i)\right)_{i=1,\ldots,n}$ be a discretized
version of a fractional Brownian motion process and let $\Vect{a}$
be a filter of length $\ell+1$ and of order $\nu\geq 1$ with real
components i.e.:
$$ \sum_{q=0}^\ell q^j a_q=0, \mbox{ for } j=0,\ldots,\nu-1 \;\; \quad  \mbox{ and } \quad \sum_{q=0}^\ell q^\nu a_q \neq 0. $$
Define also $\Vect{X}^{\Vect{a}}$ to be the series obtained by
filtering $\Vect{X}$ with $\Vect{a}$, then:
$$
X^{\Vect{a}}\left(i\right) = \sum_{q=0}^\ell a_q X\left(
i-q\right), \quad \mbox{ for } i\geq \ell+1
$$
and $\Vect{\tilde{X}}^{\Vect{a}}$ as the normalized vector of
$\Vect{X}^{\Vect{a}}$, i.e.:
$$
\Vect{\tilde{X}}^{\Vect{a}}=\frac{\Vect{X}^{\Vect{a}}}{\Esp((X^{\Vect{a}}(1))^2)^{1/2}}.
$$
It should be noticed here that the filtering operation allows to
decorrelate the increments of the discretized version of the
fractional Brownian motion process. Indeed, it may be proved (see {\it e.g.} \citet{Coeurjolly1}) that: $\rho_{H}^{a}(i)\sim
k_{H}|i|^{2H-2\nu}$ as $|i|\rightarrow +\infty$.

Consider the sequence $(\Vect{a}^m)_{m\geq 1}$ defined by:
$$
a_i^m \; = \left\{ \begin{array}{ll}\
a_j&\; \mbox{if } i=jm \\
0  &\; \mbox{otherwise } \\
\end{array} \right. \qquad\qquad \mbox{for } i=0,\ldots,m\ell \;,
$$
which is the filter $\Vect{a}$ dilated $m$ times. It has been
shown in \citet{Coeurjolly1, Coeurjolly2} that:
$$
\tilde{\Vect{X}}^{\Vect{a^{m}}}=\frac{\Vect{X}^{\Vect{a^{m}}}}{\sigma_{m}}
$$
where $ \sigma_{m}^2=m^{2H}\sigma^{2}\kappa_{H}^{\Vect{a}}$ and
$\kappa_{H}^{\Vect{a}}=\frac{-1}{2} \sum_{q,q^\prime=0}^\ell a_q
a_{q^\prime} |q-q^\prime|^{2H}$.

The following proposition allows us to construct an ordinary least
squares (OLS) estimator of the Hurst exponent $H$ of a fBm
process based on sample expectiles.

\begin{proposition}

Let $\widehat{E} \left(p;\Vect{h}(\Vect{{X}^{a^m}})\right)$ and
$\widehat{E}\left(p; \Vect{h}(\Vect{\tilde{X}^{a^m}})\right)$ be
the $p$th order sample expectiles for the filtered series
$\Vect{h}(\Vect{{X}^{a^m}})$ and
$\Vect{h}(\Vect{\tilde{X}^{a^m}})$ respectively. Here two positive
functions $h(\cdot)$ are considered, namely:
$h(\cdot)=|\cdot|^\beta$ for $\beta>0$ and
$h(\cdot)=\log|\cdot|$. We have:

\begin{equation} \label{scaling1}
\widehat{E}\left(p;{|\Vect{X}^{\Vect{a}^m}|^\beta}\right)=\sigma_{m}^{\beta}\widehat{E}\left(p;
{|\tilde{\Vect{X}}^{\Vect{a}^m}|^\beta})\right)
\end{equation}
and
\begin{equation} \label{scaling2}
\widehat{E}\left(p;
{\log|\Vect{X}^{\Vect{a}^m}|}\right)=\frac{1}{2}\log(\sigma_{m}^2)+\widehat{E}\left(p;
{\log|\Vect{\tilde{X}}^{\Vect{a}^m}|}\right).
\end{equation}
\end{proposition}

\begin{proof}

We have:
\begin{eqnarray*}
\widehat{E}\left(p;{|\Vect{X}^{\Vect{a}^m}|^\beta}\right)
&=&argmin_{\theta}{\frac{1}{n-m\ell}}\sum_{i=m\ell}^{n-1}|p-\mathbf{1}_{\{{|X^{\Vect{a}^m}(i)|^\beta}\leq
\theta\}}|.({|X^{\Vect{a}^m}(i)|^\beta}-\theta)^2\\
& =&
argmin_{\theta}{\frac{1}{n-ml}}\sum_{i=ml}^{n-1}|p-\mathbf{1}_{\{{|\tilde{X}^{\Vect{a}^m}(i)|^\beta}\leq
{\frac{\theta}{\sigma_{m}^{\beta}}}\}}|.({|\tilde{X}^{\Vect{a}^m}(i)|^\beta}-\frac{\theta}{\sigma_{m}^{\beta}})^2.
\end{eqnarray*}
Setting $\theta'=\frac{\theta}{\sigma_{m}^{\beta}}$, the proof of
the first relation~(\ref{scaling1}) follows easily. Using the same
methodology, we can demonstrate the result given by equation~(\ref{scaling2}).
\end{proof}

\begin{remark}
It should be stressed here that the scaling relationship relating
the \emph{theoretical $p$th expectiles} for the series
$\Vect{h}(\Vect{{X}^{a^m}})$ and
$\Vect{h}(\Vect{\tilde{X}^{a^m}})$ can be obtained directly using
the scale equivariance property~(\ref{Eqv1}) for
$h(\cdot)=|\cdot|^\beta$ and the location equivariance property~(\ref{Eqv2}) for $h(\cdot)=\log|\cdot|$.
\end{remark}

Now applying the logarithmic transformation to both sides of~(\ref{scaling1}), we get:
\begin{equation} \label{regression1}
\log \widehat{E}\left(p;
{|\Vect{X}^{\Vect{a}^m}|^\beta}\right)=\beta
H\log(m)+\log\left(\sigma^{\beta}(\kappa_{H}^{\Vect{a}})^{\beta/2}E_{|Y|^\beta}(p)\right)+\log \left(\frac
{\widehat{E}\left(p;
{|\Vect{\tilde{X}}^{\Vect{a}^m}|^\beta}\right)}{E_{|Y|^\beta}(p)}\right).
\end{equation}
On the other hand, (\ref{scaling2}) can be reformulated in
the following way:
\begin{equation} \label{regression2}
\widehat{E}\left(p;
{\log|\Vect{X}^{\Vect{a}^m}|}\right)=H\log(m)+\frac12\log(\sigma^{2}\kappa_{H}^{\Vect{a}}) + E_{\log|Y|}(p)+\left( \widehat{E}\left(p;
{\log(|\Vect{\tilde{X}}^{\Vect{a}^m}|})\right) -{E_{\log|Y|}(p))} \right).
\end{equation}
It is noteworthy here that we expect that $\log \widehat{E}\left(p;
{|\Vect{\tilde{X}}^{\Vect{a}^m}|^\beta}\right) / E_{|Y|^\beta}(p) $ and $\widehat{E}\left(p;{\log(|\Vect{\tilde{X}}^{\Vect{a}^m}|^\beta})\right) - E_{\log|Y|}(p)  $
to converge towards 0 as $n\rightarrow\infty$.
Hence, based on equations~(\ref{regression1}) and~(\ref{regression2}),
we opt for an OLS regression scheme. This allows to derive the two
following estimators of the hurst index defined by:
\begin{equation}\label{def-Hbeta}
\widehat{H}^{\beta}= \frac{\tr{\Vect{A}}}{\beta||\Vect{A}||^2}
\left(\log\widehat{E}\left(p;
{|\Vect{X}^{\Vect{a}^m}|^\beta}\right)\right)_{m=1,\ldots,M},
\end{equation}
and
\begin{equation}\label{def-Hlog}
\widehat{H}^{\log}= \frac{\tr{\Vect{A}}}{||\Vect{A}||^2}
\left(\widehat{E} \left(p;\log{|\Vect{X}^{\Vect{a}^m}|}\right)\right)_{m=1,\ldots,M},
\end{equation}
where $\Vect{A}$ is the vector of length $M$ with components
$A_m=\log m -\frac1M \sum_{m=1}^M\log(m)$, $m=1,\ldots,M$ for some
$M\geq 2$ whereas $||\Vect{z}||$ for some vector $\Vect{z}$ of
length $d$ designates the norm defined by $\left(\sum_{i=1}^d
z_i^2 \right)^{1/2}$. Notice here that $\widehat{H}^{\beta}$ and
$\widehat{H}^{\log}$ do not depend on $\sigma^2$.

We would like to put the stress on the fact that (\ref{def-Hbeta})
and (\ref{def-Hlog}) are really similar to the ones developed in
\citet{Coeurjolly1,Coeurjolly2}. Indeed, the standard procedure
developed in~\citet{Coeurjolly1} simply consists in replacing the
sample expectile by the sample variance (this method will be
denoted by ST in Section~\ref{sec-simuls}). To deal with outliers, the
procedure developed in \citet{Coeurjolly2} consists in replacing
the sample expectile by either the sample median of
$(\Vect{X}^{\Vect{a}^m})^2$ or the trimmed-means of
$(\Vect{X}^{\Vect{a}^m})^2$. These two last methods are denoted by
MED and TM in Section~\ref{sec-simuls}.

Now, we state the asymptotic results for these new estimates based on expectiles.

\begin{proposition}\label{prop-H} Let $\Vect{a}$ a filter with order $\nu \geq 2$, $p \in (0,1)$, $\beta>0$ then as $n \to +\infty$, $\widehat{H}^\beta$ and $\widehat{H}^{\log}$ converge in probability to $H$. Moreover, the following convergences in distribution hold
$$
\sqrt{n} \left( \widehat{H}^\beta-H \right)  \stackrel{d}{\longrightarrow}  \mathcal{N}(0,\sigma^2_\beta) \quad \mbox{ and } \quad
\sqrt{n} \left( \widehat{H}^{\log}-H \right)  \stackrel{d}{\longrightarrow}  \mathcal{N}(0,\sigma^2_{\log}),$$
where
$$
\sigma_\beta^2 = \frac1{E_{|Y|^\beta}(p)^2} \;\times \;  \frac{\Vect{A}^T \Mat{\Sigma}^\beta \Vect{A}}{\beta^2 \|\Vect{A}\|^4}
\quad \mbox{ and } \quad
\sigma^2_{\log} = \frac{\Vect{A}^T \Mat{\Sigma}^{\log} \Vect{A}}{\| \Vect{A}\|^4}
$$
and where the $(M,M)$ matrices $\Mat{\Sigma}^\beta$ and $\Mat{\Sigma}^{\log}$ are defined by~(\ref{def-Sigma}).
\end{proposition}

\begin{proof}
We only provide a sketch of the proof. We claim that once Theorem~\ref{thm-bahadur} and Corollary~\ref{cor-bahadur} are established, the obtention of convergences stated in Proposition~\ref{prop-H} are semi-routine. First of all, let us notice that
\begin{equation}\label{HbetaH}
\widehat{H}^\beta -H = \frac{\Vect{A}^T}{\beta\|\Vect{A}\|^2} \left(\log \left( \frac{\widehat{E}\left(p;{|\widetilde{\Vect{X}}^{\Vect{a}^m}|^\beta} \right)}{E_{|Y|^\beta}}
\right)\right)_{m=1,\ldots,M}
\end{equation}
and
\begin{equation}\label{HlogH}
\widehat{H}^{\log}-H = \frac{\Vect{A}^T}{\beta\|\Vect{A}\|^2}  \left(\widehat{E} \left(p;\log{|\widetilde{\Vect{X}}^{\Vect{a}^m}|}\right) - E_{\log|Y|}(p) \right)_{m=1,\ldots,M}.
\end{equation}
Since the functions $|\cdot|^\beta$ and $\log|\cdot|$ have Hermite rank~2 and since the correlation function of the stationary sequence $\widetilde{\Vect{X}}^{\Vect{a}^m}$ decreases hyperbolically with an exponent $\alpha=2\nu -2H$ then for any $m\in \{1,\ldots,M\}$, Theorem~\ref{thm-bahadur} holds with $r_n=n^{-1/2}$ (since $\alpha \tau >1$ for all $H\in (0,1)$). This ensures the convergence in probability of the new estimates. \\

The cross-correlation between $\widetilde{X}^{\Vect{a}^{m_1}}$ and $\widetilde{X}^{\Vect{a}^{m_2}}$ is defined by
$$
\rho_H^{\Vect{a}^{m_1},\Vect{a}^{m_2}}(j) = \frac{\pi_H^{\Vect{a}^{m_1},\Vect{a}^{m_2}}(j) }{\pi_H^{\Vect{a}^{m_1},\Vect{a}^{m_1}}(0)^{1/2} \pi_H^{\Vect{a}^{m_2},\Vect{a}^{m_2}}(0)^{1/2} } \;\;\mbox{ with }\;\;
\pi_H^{\Vect{a}^{m_1},\Vect{a}^{m_2}}(j) = \sum_{q,r=0}^{\ell} a_q a_r|m_1 q - m_2 r +j|^{2H}.
$$
Lemma~1 in \citet{Coeurjolly2} states that for all $m_1,m_2$ the correlation function $\rho_H^{\Vect{a}^{m_1},\Vect{a}^{m_2}}$ is also decreasing hyperbolically with an exponent $\alpha=2\nu-2H$, then Corollary~\ref{cor-bahadur} may be applied to prove that
\begin{equation}\label{mult-beta}
\left( \widehat{E}\left(p;{|\widetilde{\Vect{X}}^{\Vect{a}^m}|^\beta} \right) - E_{|Y|^\beta} \right)_{m=1,\ldots,M}
\stackrel{d}{\longrightarrow} \mathcal{N}(0,\Mat{\Sigma}^\beta)
\end{equation}
and
\begin{equation}\label{mult-log}
\left( \widehat{E}\left(p;{\log|\widetilde{\Vect{X}}^{\Vect{a}^m}|} \right) - E_{\log|Y|} \right)_{m=1,\ldots,M}
\stackrel{d}{\longrightarrow} \mathcal{N}(0,\Mat{\Sigma}^{\log}),
\end{equation}
where according to~(\ref{def-Sigma}), the $(M,M)$ matrices $\Mat{\Sigma}^\beta$ and $\Mat{\Sigma}^{\log}$ are respectively defined by
\begin{eqnarray}
\Sigma_{m_1m_2}^\beta = \frac1{\psi^\prime\left( E_{|Y|^\beta}(p);p \right)^2} \;
\sum_{i\in \ZZ} \sum_{k\geq 2} \frac{c_k^{\beta}(p)^2 }{k!}
\rho_H^{\Vect{a}^{m_1},\Vect{a}^{m_2}}(i)^k \\
\Sigma_{m_1m_2}^{\log} = \frac1{\psi^\prime\left( E_{\log|Y|}(p);p \right)^2} \;
\sum_{i\in \ZZ} \sum_{k\geq 2} \frac{c_k^{\log}(p)^2 }{k!}
\rho_H^{\Vect{a}^{m_1},\Vect{a}^{m_2}}(i)^k,
\end{eqnarray}
where $(c_k^\beta)_{k\geq 2}$ and $(c_k^{\log})_{k\geq 2}$ are respectively the Hermite coefficients of the functions $|\cdot|^\beta$ and $\log|\cdot|$. The convergences~(\ref{mult-beta}) and~(\ref{mult-log}) combined with~(\ref{HbetaH}) and~(\ref{HlogH}) and the use of the delta-method (for the convergence of $\widehat{H}^{\beta}$) end the proof.
\end{proof}

\subsection{A short simulation study} \label{sec-simuls}

In this section, we investigate the interest of the new estimators
based on expectiles. We consider three different models in our
simulations.
\begin{itemize}
\item[(a)] {\it standard fBm}: non-contaminated fractional Brownian motion.
\item[(b)] {\it fBm with additive outliers}: we contaminate $5\%$ of the observations of the increments of the fractional Brownian motion with an independent Gaussian noise such that the SNR of the considered components equals $-20Db$.
\item[(c)] {\it rounded fBm}: we assume the data are given by the integer part of a discretized sample path of an original fBm.
\end{itemize}
To fix ideas, Figure~\ref{exCont} provides some examples of
discretized sample paths of standard and contaminated fBm. The simulation results are presented in
Tables~\ref{tab-H02} and~\ref{tab-H08}. For these simulations, as
suggested in \citet{Coeurjolly1}, we chose the filter $a=d4$
corresponding to the wavelet Daubechies filter with order~4 (see
\citet{Daubechies}) and the maximum number of dilated filters
$M=5$. Also, in other simulations not presented here, we have
observed that the estimates $\widehat{H}^\beta$ perform better
than $\widehat{H}^{\log}$ and, among all possible choices of
$\beta$, the value $\beta=2$ seems to be a good compromise.
Therefore, we present only the result for this latter estimator,
that is $\widehat{H}^\beta$ with $\beta=2$.

In a first step, we had observed a quite large sensitivity to the
value of the probability $p$ defining the expectile. In order to
have an efficient data-driven procedure, we propose to choose the
probability parameter $p$ via a Monte-Carlo approach as follows:
\begin{enumerate}
\item Estimate the parameters $H$ and $\sigma^2$ using the standard method (the estimation of $\sigma^2$ is not described here but it may be found for example in \citet{Coeurjolly1}). Denote these estimates $\widehat{H}_0$ and $\widehat{\sigma}^2_0$.
\item Generate $B=100$ contaminated fBm with Hurst parameter $\widehat{H}_0$ and scaling coefficient $\widehat{\sigma}^2_0$, define a grid of probabilities $(p_1,\ldots,p_P)$. For each new replication, we estimate $\widehat{H}_0$ with expectiles for all the $p_i$. The optimal $p$, denoted in the tables by $p^{opt}$, is then defined as the one achieving the smallest mean squared error (based on the $B=100$ replications).
\end{enumerate}
The procedure based on expectiles, denoted E(p) in the results, is compared to the standard method (ST) and to methods which efficiently deal with outliers, that is methods MED and TM (the last one is calculated by discarding $5\%$
 of the lowest and the highest values of $(\Vect{X}^{\Vect{a}^m})^2$ at each scale $m$).

The {\it standard fBm} model is used as a control to show that all
methods perform well. As seen in the first two columns of
Tables~\ref{tab-H02} and~\ref{tab-H08}, this is indeed true. All
the methods seem to be asymptotically unbiased and have a variance
converging to zero. We can also remark that in this situation
whatever the value of $H$, estimates based on expectiles exhibit a
performance which is very close to the one of the standard method
(wich can also be viewed as the method based on expectile with
$p=0.5$). Several types of expectiles are investigated. When the
discretized sample path of the fBm is contaminated by outliers, we
recover the results already shown in \citet{Coeurjolly2},
\citet{Achard} or \citet{Kouamo}: methods based on medians or
trimmed-means are very efficient which is in agreement with the
fact that quantiles have a finite gross error sensitivity. The
inefficiency of expectiles for such a contamination is also
coherent since expectiles have infinite gross error sensitivity.
Finally, the interest of the expectile-based method can be seen
with the {\it rounded fBm} corresponding to the last two columns
of each table. In this situation, expectiles are shown to be more
efficient in terms of bias and its variance seems to be not too
much affected by this type of strong contamination. We also put
the stress on the interest and efficiency to choose the $p$ value
based on a Monte-Carlo approach.

\begin{table}[ht]
{\small \hspace*{-1cm}\begin{tabular}{r|rrrrrr}
  \hline
&\multicolumn{2}{c}{Standard fBm} & \multicolumn{2}{c}{fBm with additive outliers} &  \multicolumn{2}{c}{Rounded fBm} \\
 & $n=500$ & \multicolumn{1}{r}{$n=5000$} & $n=500$ & \multicolumn{1}{r}{$n=5000$} &$n=500$ & \multicolumn{1}{r}{$n=5000$}  \\
  \hline
$E(p=0.2)$ & 0.198 (0.033) & 0.200 (0.011) & 0.280 (0.062) & 0.298 (0.024) & 0.334 (0.036) & 0.337 (0.011) \\
 $E(p=0.4)$ & 0.198 (0.032) & 0.200 (0.010) & 0.288 (0.068)& 0.309 (0.026) & 0.298 (0.034) & 0.300 (0.011) \\
 $E(p=0.6)$ & 0.199 (0.032) & 0.200 (0.010) & 0.298 (0.076) & 0.323 (0.029) & 0.284 (0.035) & 0.287 (0.011) \\
 $E(p=0.8)$ & 0.199 (0.033) & 0.200 (0.010) & 0.311 (0.086) & 0.349 (0.034) & 0.275 (0.037) & 0.277 (0.011) \\
 $E(p=p^{opt})$ & 0.199 (0.035) & 0.200 (0.011) & 0.314 (0.085) & 0.368 (0.033) & {\bf 0.249 (0.040)} & {\bf 0.240 (0.012)} \\
\hline
 MED & 0.197 (0.048) & 0.200 (0.016) & 0.227 (0.050) & 0.227 (0.016) & 0.451 (0.158) & 0.361 (0.119) \\
 TM & 0.206 (0.034) & 0.201 (0.011) & {\bf 0.222 (0.052)} & {\bf0.225 (0.016)} & 0.294 (0.038) & 0.289 (0.012) \\
\hline
 ST & {\bf 0.199 (0.032)} & {\bf 0.200 (0.010)} & 0.293 (0.072) & 0.315 (0.027) & 0.290 (0.034) & 0.292 (0.011) \\
   \hline
\end{tabular}}

\caption{Empirical means and standard deviations of $H$ estimates
based on $m=500$ replications of non-contaminated and contaminated
fractional Brownian motions with scale parameter $\sigma=.5$,
Hurst parameter {\bf $H=0.2$} and sample size $n=500,5000$ are
given between brackets. Methods based on expectiles, quantiles and
trimmed-means as well as the standard method are considered. The
filter $a$ correspond to the Daubechies wavelet filter with order
4 (two zero moments) and we set $M_1=1,M_2=5$. According to a
sample size and a model, the method achieving the lowest mean
squared error is printed in bold.\label{tab-H02}}
\end{table}

\begin{table}[ht]
{\small \hspace*{-1cm}\begin{tabular}{r|rrrrrr}
  \hline
&\multicolumn{2}{c}{Standard fBm} & \multicolumn{2}{c}{fBm with additive outliers} &  \multicolumn{2}{c}{Rounded fBm} \\
 & $n=500$ & \multicolumn{1}{r}{$n=5000$} & $n=500$ & \multicolumn{1}{r}{$n=5000$} &$n=500$ & \multicolumn{1}{r}{$n=5000$}  \\
  \hline
$E(p=0.2)$ & 0.796 (0.044) & 0.800 (0.014) & 0.725 (0.065) & 0.715 (0.024) & 0.775 (0.049) & 0.777 (0.014)\\
 $E(p=0.4)$ & 0.795 (0.043) & 0.800 (0.013) & 0.712 (0.072) & 0.700 (0.027) & 0.725 (0.048) & 0.728 (0.014) \\
 $E(p=0.6)$ & 0.795 (0.042) & 0.800 (0.013) & 0.692 (0.083) & 0.679 (0.032) & 0.702 (0.048) & 0.706 (0.013) \\
 $E(p=0.8)$ &  0.794 (0.043) & 0.800 (0.014) & 0.653 (0.105) & 0.632 (0.042) & 0.702 (0.049) & 0.707 (0.014)\\
 $E(p=p^{opt})$ & 0.796 (0.045) & 0.800 (0.014) & 0.722 (0.073) & 0.713 (0.027) & {\bf 0.786 (0.058)} & {\bf 0.786 (0.018)} \\
\hline
 MED &  0.799 (0.064) & 0.800 (0.019) & 0.817 (0.062) & 0.816 (0.021) & 1.242 (2.487) & 0.903 (0.017) \\
 TM &  0.803 (0.045) & 0.801 (0.014) & {\bf 0.815 (0.052)} & {\bf 0.814 (0.017)} & 0.696 (0.048) & 0.689 (0.014) \\
\hline
 ST &{\bf 0.798 (0.042)} & {\bf 0.800 (0.013)} & 0.703 (0.077) & 0.691 (0.029) & 0.710 (0.048) & 0.714 (0.013) \\
   \hline
\end{tabular}}
\caption{Empirical means and standard deviations of $H$ estimates
based on $m=500$ replications of non-contaminated and contaminated
fractional Brownian motions with scale parameter $\sigma=.5$,
Hurst parameter {\bf $H=0.8$} and sample size $n=500,5000$ are
given between brackets. Methods based on expectiles, quantiles and
trimmed-means as well as the standard method are considered. The
filter $a$ correspond to the Daubechies wavelet filter with order
4 (two zero moments) and we set $M_1=1,M_2=5$. According to a
sample size and a model, the method achieving the lowest mean
squared error is printed in bold.\label{tab-H08} }
\end{table}

\begin{figure}
\begin{center}
\begin{tabular}{ll}
\includegraphics[scale=.16]{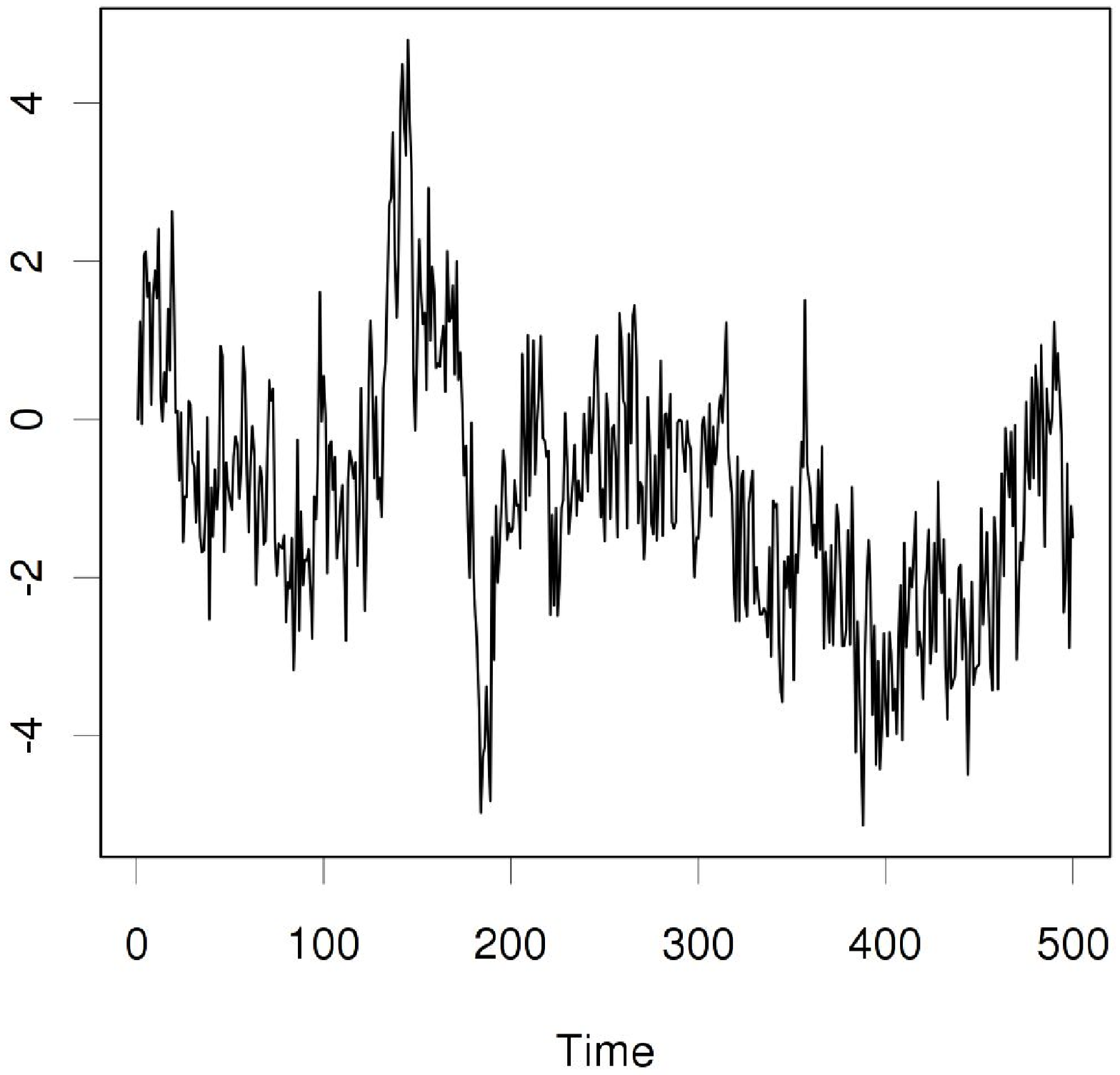} & \includegraphics[scale=.16]{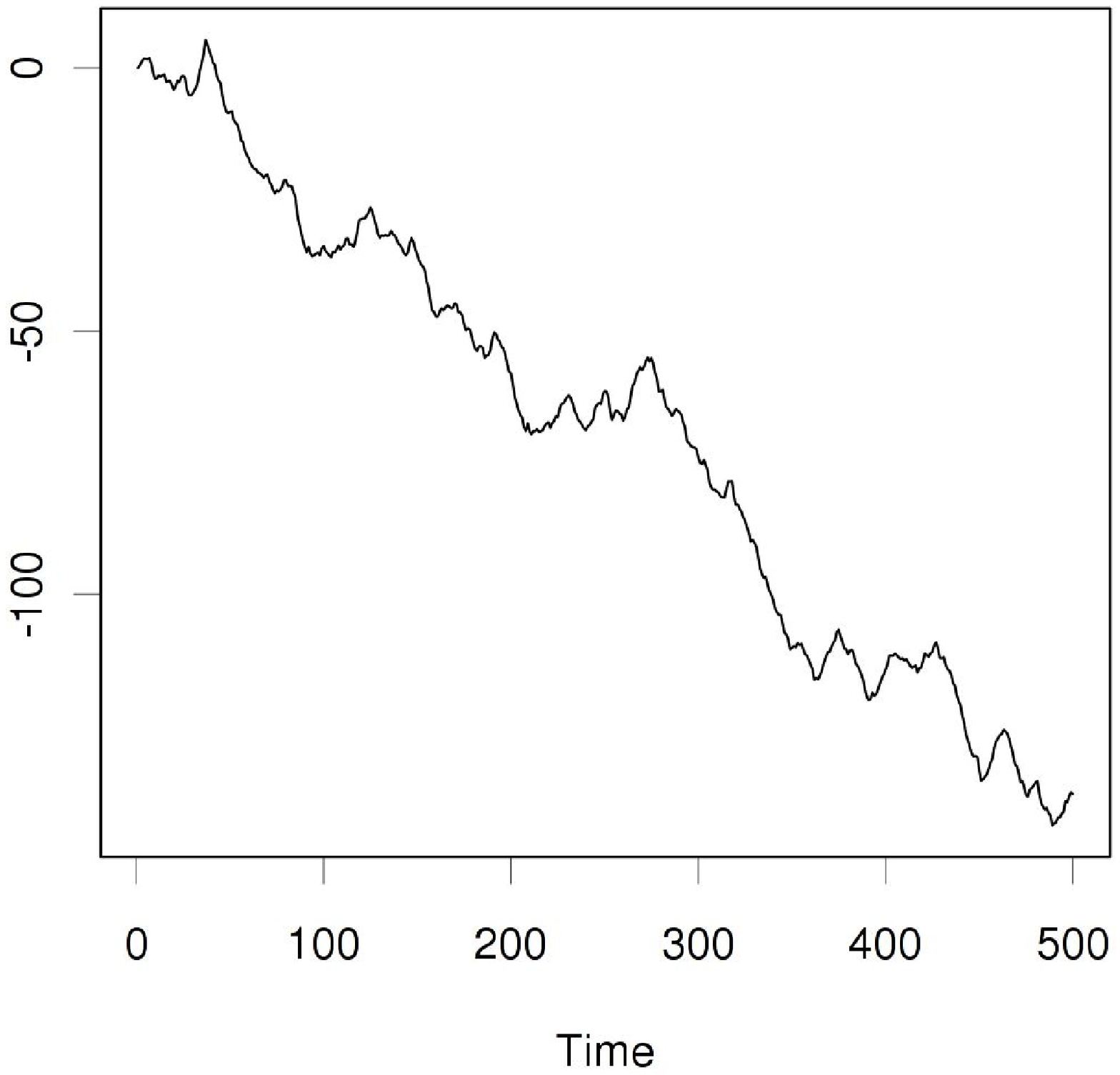} \\
\includegraphics[scale=.16]{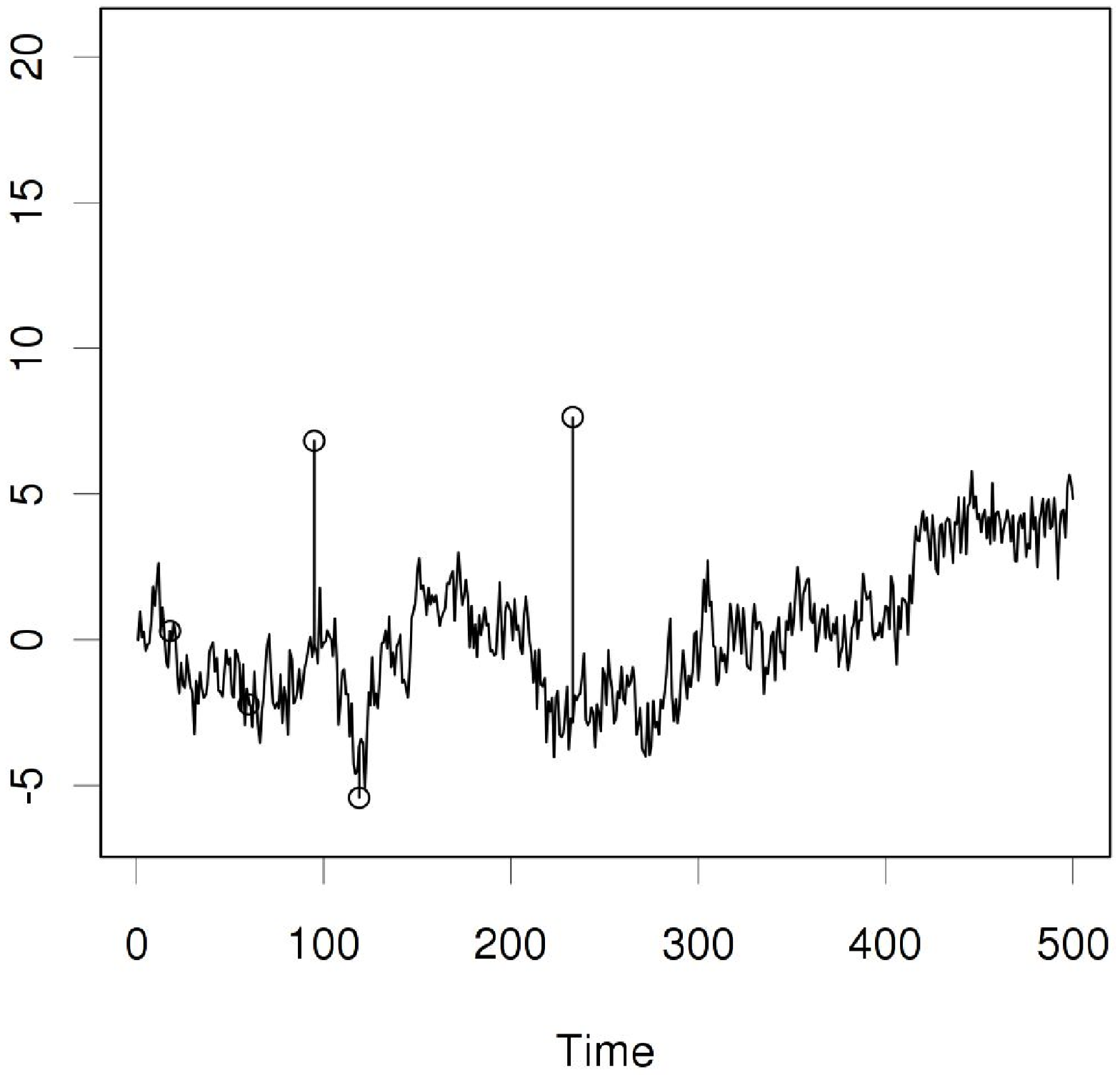} & \includegraphics[scale=.16]{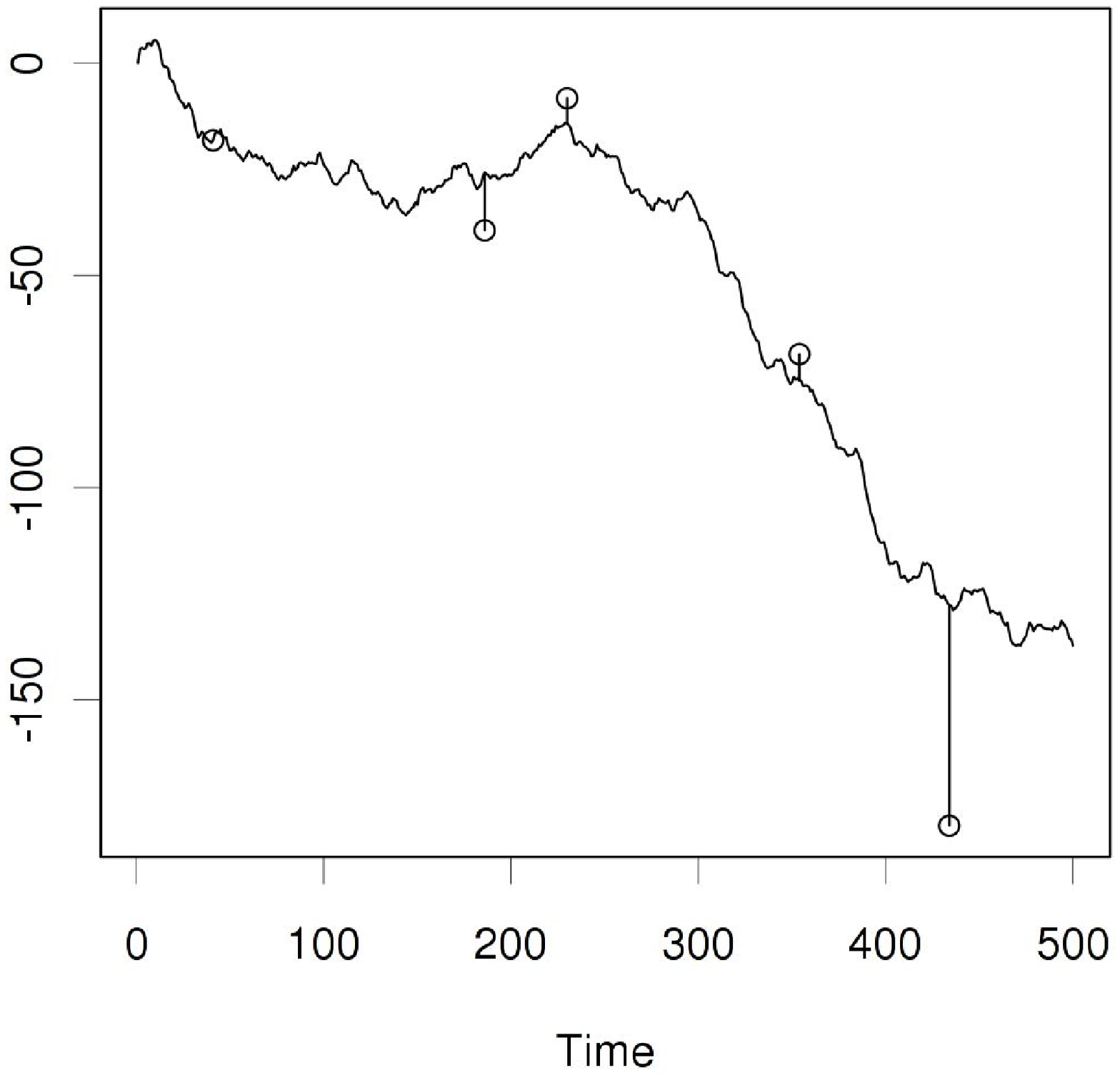} \\
\includegraphics[scale=.16]{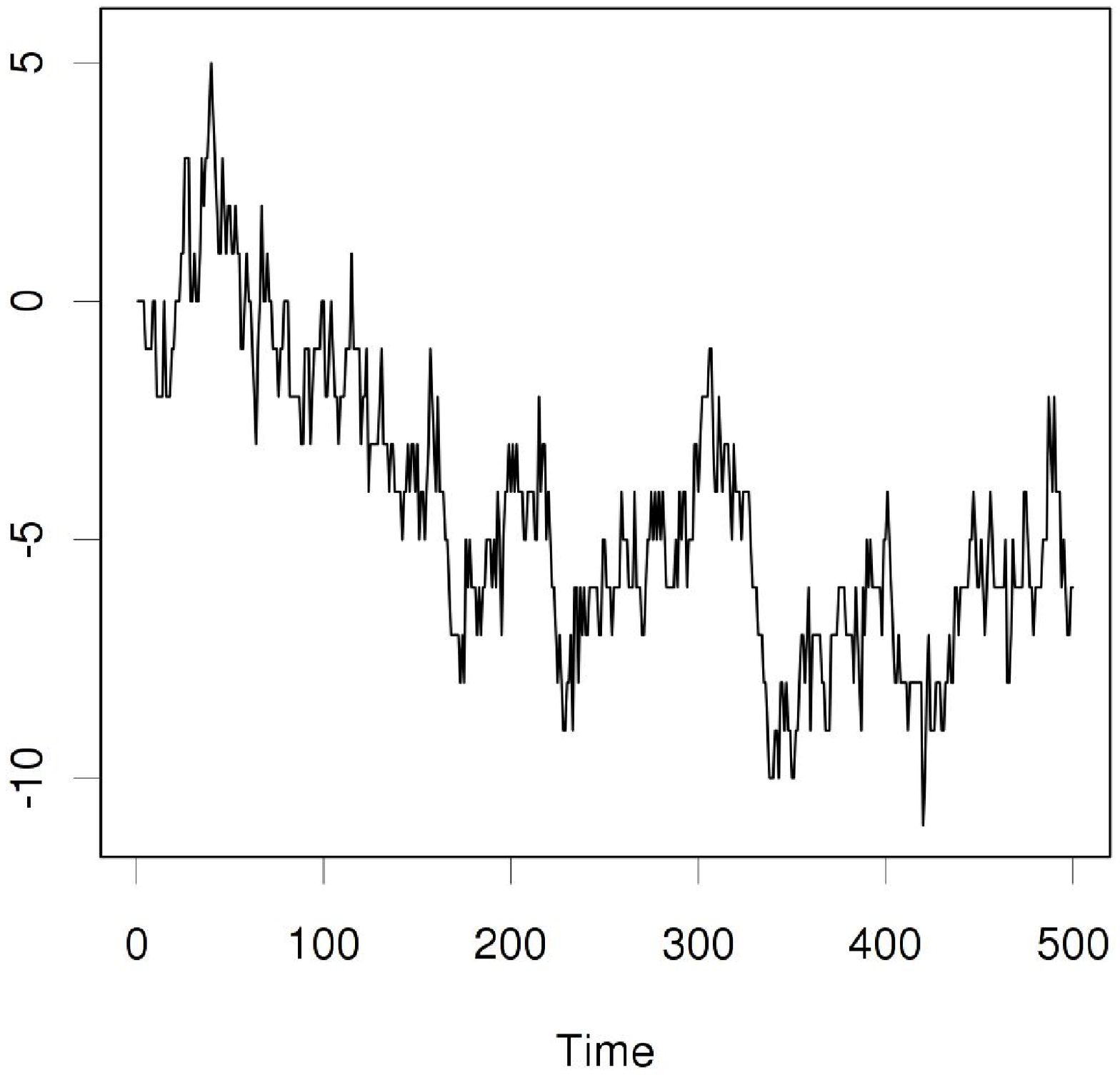} & \includegraphics[scale=.16]{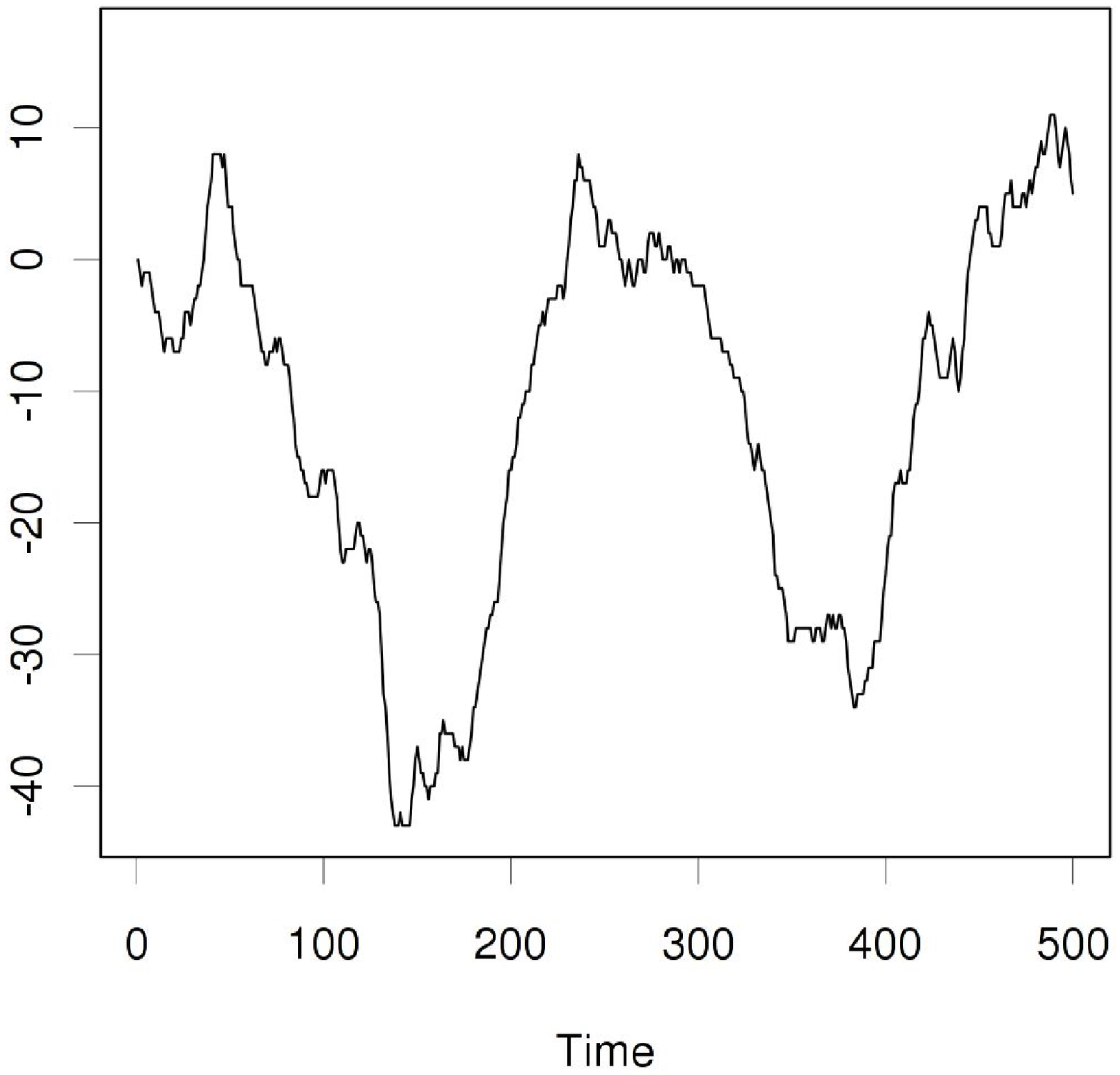}
\end{tabular}
 \caption{Examples of discretized sample path of
standard fBm (top), fBm with additive outliers (middle) and
rounded fBm (bottom) for $n=500$ and with Hurst parameters $H=0.2$
(left) and $H=0.8$ (right). \label{exCont}}
\end{center}
\end{figure}

\footnotesize
\bibliographystyle{plainnat}  
\bibliography{expectile}

\begin{thebibliography}{32}
\providecommand{\natexlab}[1]{#1}
\providecommand{\url}[1]{\texttt{#1}}
\expandafter\ifx\csname urlstyle\endcsname\relax
  \providecommand{\doi}[1]{doi: #1}\else
  \providecommand{\doi}{doi: \begingroup \urlstyle{rm}\Url}\fi

\bibitem[Abry et~al.(2000)Abry, Flandrin, Taqqu, and Veitch]{Abry2}
P.~Abry, P.~Flandrin, M.~S. Taqqu, and D.~Veitch.
\newblock \emph{Wavelets for the analysis, estimation, and synthesis of scaling
  data}, pages 39--88.
\newblock In Self-similar Network Traffic and Performance Evaluation. by K.
  Park and W. Willinger, Wiley, New York, 2000.

\bibitem[Abry et~al.(2003)Abry, Flandrin, Taqqu, and Veitch]{Abry3}
P.~Abry, P.~Flandrin, M.S. Taqqu, and D.~Veitch.
\newblock \emph{Self-similarity and long-range dependence through the wavelet
  lens}, pages 527--556.
\newblock Theory and applications of long-range dependence. Birkhäuser, 2003.

\bibitem[Achard and Coeurjolly(2010)]{Achard}
S.~Achard and J-F. Coeurjolly.
\newblock Discrete variations of the fractional brownian motion in the presence
  of outliers and an additive noise.
\newblock \emph{Statistics Surveys}, 4:\penalty0 117--147, 2010.

\bibitem[Arcones(1994)]{Arcones}
M.A. Arcones.
\newblock Limit theorems for nonlinear functionals of stationary gaussian field
  of vectors.
\newblock \emph{Ann. Probab.}, 22:\penalty0 2242--2274, 1994.

\bibitem[Bardet et~al.(2000)Bardet, Lang, Moulines, and Soulier]{BardLMS00}
J.M. Bardet, G.~Lang, E.~Moulines, and P.~Soulier.
\newblock Wavelet estimator of long-range dependent processes.
\newblock \emph{Statistical Inference for Stochastic Processes}, 3:\penalty0
  85--99, 2000.

\bibitem[Beran(1994)]{Beran}
J.~Beran.
\newblock \emph{Statistics for long memory processes}.
\newblock Monogr. Stat. Appl. Probab. 61. Chapman and Hall, London, 1994.

\bibitem[Bijwaard and Franses(2009)]{Bijwaard}
G.E. Bijwaard and P.~Hans Franses.
\newblock The effect of rounding on payment efficiency.
\newblock \emph{Computational statistics and data analysis}, 53\penalty0
  (4):\penalty0 1449--1461, 2009.

\bibitem[Bois and Vignes(1982)]{Bois}
P.~Bois and J.~Vignes.
\newblock An algorithm for automatic round-off error analysis in discrete
  linear transforms.
\newblock \emph{International Journal of Computer Mathematics}, 12\penalty0
  (2):\penalty0 161--171, 1982.

\bibitem[Coeurjolly(2000)]{Coeurjolly0}
J.-F. Coeurjolly.
\newblock Simulation and identification of the fractional brownian motion: a
  bibliographical and comparative study.
\newblock \emph{J. Stat. Softw.}, 5\penalty0 (7):\penalty0 1--53, 2000.

\bibitem[Coeurjolly(2001)]{Coeurjolly1}
J.-F. Coeurjolly.
\newblock Estimating the parameters of a fractional brownian motion by discrete
  variations of its sample paths.
\newblock \emph{Stat. Infer. Stoch. Process.}, 4\penalty0 (2):\penalty0
  199--227, 2001.

\bibitem[Coeurjolly(2008)]{Coeurjolly2}
J.-F. Coeurjolly.
\newblock Hurst exponent estimation of locally self-similar gaussian processes
  using sample quantiles.
\newblock \emph{Annals of Statistics}, 36\penalty0 (3):\penalty0 1404--1434,
  2008.

\bibitem[Daubechies(1992)]{Daubechies}
I.~Daubechies.
\newblock \emph{Ten lectures on wavelets}.
\newblock CBMS-NSF Regional Conference Series on Applied Mathematics, 61. SIAM,
  Philadelphia, 1992.

\bibitem[Fa{\"y} et~al.(2009)Fa{\"y}, Moulines, Roueff, and Taqqu]{FayMRT09}
G.~Fa{\"y}, E.~Moulines, F.~Roueff, and M.S. Taqqu.
\newblock Estimators of long-memory: Fourier versus wavelets.
\newblock \emph{Journal of econometrics}, 151\penalty0 (2):\penalty0 159--177,
  2009.

\bibitem[Flandrin(1992)]{Flandrin}
P.~Flandrin.
\newblock Wavelet analysis and synthesis of fractional brownian motion.
\newblock \emph{IEEE Trans. Inform. Theory.}, 38\penalty0 (2, part 2):\penalty0
  910--917, 1992.

\bibitem[Ghosh(1971)]{Ghosh}
J.K. Ghosh.
\newblock A new proof of the bahadur representation of quantiles and an
  application.
\newblock \emph{The Annals of Mathematical Statistics}, 42\penalty0
  (6):\penalty0 1957--1961, 1971.

\bibitem[Hampel et~al.(1986)Hampel, Ronchetti, Rousseeuw, and Stahel]{Hampel}
F.~R. Hampel, E.~M. Ronchetti, P.~J. Rousseeuw, and W.~A. Stahel.
\newblock \emph{Robust Statistics: The Approach Based on Influence Functions}.
\newblock Wiley, New York, 1986.

\bibitem[Huber(1981)]{Huber}
P.~J. Huber.
\newblock \emph{Robust statitics}.
\newblock Wiley series in probability and mathematical statistics. Wiley, 1981.

\bibitem[Istas and Lang(1997)]{Istas}
J.~Istas and G.~Lang.
\newblock Quadratic variations and estimation of the holder index of a gaussian
  process.
\newblock \emph{Ann. Inst. H. Poincar\'e Probab. Statist.}, 33:\penalty0
  407--436, 1997.

\bibitem[Kent and Wood(1997)]{Kent}
J.T. Kent and A.T.A. Wood.
\newblock Estimating the fractal dimension of a locally self-similar gaussian
  process using increments.
\newblock \emph{J. Roy. Statist. Soc. Ser. B}, 59:\penalty0 679--700, 1997.

\bibitem[Kouamo et~al.(2010)Kouamo, L\'evy-Leduc, and Moulines]{Kouamo}
O.~Kouamo, C.~L\'evy-Leduc, and E.~Moulines.
\newblock Central limit theorem for the robust log-regression wavelet
  estimation of the memory parameter in the gaussian semi-parametric context,
  2010.
\newblock arXiv:1011.4370.

\bibitem[Mandelbrot and Ness(1968)]{Mandelbrot}
B.~B. Mandelbrot and J.~W.~Van Ness.
\newblock Siam review.
\newblock \emph{IEEE Transactions on Pattern Analysis and Machine
  Intelligence}, 10:\penalty0 422--437, 1968.

\bibitem[Matthieu(2006)]{Matthieu}
M.~Matthieu.
\newblock Semantics of roundoff error propagation in finite precision
  calculations.
\newblock \emph{Higher-order and symbolic computation}, 19\penalty0
  (1):\penalty0 7--30, 2006.

\bibitem[Newey and Powell(1987)]{Newey}
W.K. Newey and J.L. Powell.
\newblock Asymmetric least squares estimation and testing.
\newblock \emph{Econometrica}, 55:\penalty0 819--847, 1987.

\bibitem[Percival and Walden(2000)]{Percival}
D.~B. Percival and A.~T. Walden.
\newblock \emph{Wavelet Methods for Time Series Analysis}.
\newblock Cambridge University Press, 2000.

\bibitem[Robinson(1995)]{Robinson}
P.~Robinson.
\newblock Gaussian semiparametric estimation of long range dependence.
\newblock \emph{Ann. Stat.}, 23:\penalty0 1630--1661, 1995.

\bibitem[Rosenbaum(2009)]{Rosenbaum}
M.~Rosenbaum.
\newblock Integrated volatility and round-off error.
\newblock \emph{Bernoulli}, 15\penalty0 (3):\penalty0 687--720, 2009.

\bibitem[Soltani et~al.(2004)Soltani, Simard, and Boichu]{Soltani}
S.~Soltani, P.~Simard, and D.~Boichu.
\newblock Estimation of the self-similarity parameter using the wavelet
  transform.
\newblock \emph{Signal Process.}, 84:\penalty0 117--123, 2004.

\bibitem[Taqqu(1977)]{Taqqu}
M.S. Taqqu.
\newblock Law of the iterated logarithm for sums of non-linear functions of
  gaussian variables that exhibit a long range dependence.
\newblock \emph{Z. Wahrscheinlichkeitstheorie verw. Geb.}, 40:\penalty0
  203--238, 1977.

\bibitem[Vilmart(2008)]{Vilmart}
G.~Vilmart.
\newblock Reducing round-off errors in rigid body dynamics.
\newblock \emph{Journal of Computational Physics}, 227\penalty0 (15):\penalty0
  7083--7088, 2008.

\bibitem[Williams(2006)]{Williams}
E.~Williams.
\newblock The effects of rounding on the consumer price index.
\newblock \emph{Monthly labor review}, 129\penalty0 (10):\penalty0 80--89,
  2006.

\bibitem[Wood and Chan(1994)]{WoodChan}
A.T.A. Wood and G.~Chan.
\newblock Simulation of stationary gaussian processes in $[0,1]^d$.
\newblock \emph{J. Comput. Graph. Statist.}, 3:\penalty0 409--432, 1994.

\bibitem[Wornell and Oppenheim(1992)]{Wornell}
G.~Wornell and A.~Oppenheim.
\newblock Estimation of fractal signals from noise measurements using wavelets.
\newblock \emph{IEEE Transactions on Signal Processing}, 40:\penalty0 611--623,
  1992.

\end{thebibliography}

\end{document}